\documentclass[11pt,leqno]{article}
\usepackage[sc]{mathpazo}
\usepackage[OT2,T1]{fontenc}
\usepackage[utf8]{inputenc}
\usepackage[top = 3.5cm, bottom  = 3cm, left = 3.7cm, right = 3.2cm]{geometry}
\usepackage{amsmath, amscd, amssymb, amsthm, amsfonts, euscript, mathtools, tikz-cd}
\usepackage{adjustbox}
\usepackage{url}
\usepackage{hyperref}
\hypersetup{ colorlinks=true, citecolor= black,linkcolor = black}

\usepackage{todonotes}
\usepackage{musicography}

\DeclareSymbolFont{cyrletters}{OT2}{wncyr}{m}{n}
\DeclareMathSymbol{\Sha}{\mathalpha}{cyrletters}{"58}
\DeclareMathSymbol{\Bcyr}{\mathalpha}{cyrletters}{"42}

\renewcommand{\P}{\mathbf{P}}

\theoremstyle{definition}
\newtheorem{thm}{Theorem}[section]
\newtheorem{defi}[thm]{Definition}
\newtheorem*{defi*}{Definition}
\newtheorem{prop}[thm]{Proposition}
\newtheorem*{prop*}{Proposition}
\newtheorem{lemma}[thm]{Lemma}
\newtheorem{cor}[thm]{Corollary}
\newtheorem*{conjecture}{Conjecture}
\newtheorem*{cor*}{Corollary}
\newtheorem{rmk}[thm]{Remark}
\newcommand{\NN}{\mathbf{N}}
\newcommand{\ZZ}{\mathbf{Z}}

\newcommand{\KK}{\mathrm{K}}

\newcommand{\Norm}{\operatorname{N}}

\newcommand{\lau}[1]{(\!(#1)\!)}
\DeclareMathOperator{\Spec}{Spec}

\newtheorem{thmalpha}{Theorem}

\title{Stability of the Kato—Kuzumaki's properties under field extensions}
\author{ Felipe Gambardella \\ \vspace*{-1ex} \small \textit{École polytechnique de Paris (CMLS)} \and Harry C. Shaw  \\ \vspace*{-1ex} \small \textit{University of Bath}}
\date{}

\begin{document}

\maketitle
\begin{abstract}
\noindent In 1986, Kato and Kuzumaki introduced some Diophantine properties of fields, called the $C_i^q$ properties, and they hoped they would provide a good characterization of the cohomological dimension of fields. In this paper, we study the stability of some variants of the $C_i^q$ properties under transcendental and algebraic extensions. As an application, we obtain the $C_n^1$ property for the field $\mathbf{F}_p(x_1,\hdots,x_n)$.
\end{abstract}

\setcounter{tocdepth}{1}
\tableofcontents
\section{Introduction}
\noindent Among the many measures of complexity for a field $k$ we find the \textit{cohomological dimension}, usually denoted as $\mathrm{cd} (k)$. It is a very coarse measure for the complexity of the Galois cohomology of a field (see \cite[Definition.~6.5.9]{NeukirchSchmidtWingberg2008CohomologyNumbers}), in particular, it only depends on the absolute Galois group of $k$. Galois cohomology has a very robust theory that makes possible the computation of the cohomological dimension in most cases. However, the relationship between cohomological dimension and the \textit{Diophantine} properties of the field $k$ are still rather mysterious. For instance, it is a prominent open question due to Serre whether a $C_i$ field $k$, as defined by Tsen in \cite{Tsen1936} (see also \cite[page~374]{Lang1952QuasiAlgebraicClosure}), satisfies $\mathrm{cd} (k)\leq i$. This question admits a positive answer when $i$ is at most $2$, the case $i=1$ being rather elementary and the case $i=2$ relying on a deep result due to Suslin (see \cite[II \S 4.5]{Serre1994CohomologieGaloisienne}). Already the case of $i=3$ is still open.\par
Exploring further the relationship between rational points on hypersurfaces of small degree and the cohomological dimension, Kato and Kuizumaki suggested in \cite{KatoKuzumaki1986Dimension} a weakening of the $C_i$ property involving Milnor $\mathrm{K}$-theory. The main conjecture of \cite{KatoKuzumaki1986Dimension} states that the cohomological dimension of a field is characterised by that weakening of the $C_i$ property. Unfortunately, this conjecture is false in full generality, but still open for \textit{arithmetically interesting fields}. \par 
    
   We proceed to describe the properties defined by Kato and Kuzumaki, and their conjecture in more detail. Let $k$ be a field. For any finite extension $l/k$, Milnor $\mathrm{K}$-theory admits a norm map $\mathrm{N}_{l/k}: \mathrm{K}_q(l) \to \mathrm{K}_q(k)$, see \cite[\S 7.3]{GilleSzamuely2017CentralGalois}. For a $k$-variety $Z$, Kato and Kuzumaki defined its norm group $\mathrm{N}_q(Z/k)$ as the subgroup of $\mathrm{K}_q(k)$ generated by $\mathrm{N}_{l/k}(\mathrm{K}_q(l))$, where $l/k$ ranges over all finite extensions such that $Z(l) \neq \emptyset$. Fix a pair of non-negative integers $i,q$. The $C_i^q$ property is defined as follows: for every pair of integers $d, n\geq 1$, finite extension $k'/k$, and hypersurface $Z \subseteq \mathbf{P}^n_{k'}$ of degree $d$ such that $d^i \leq n$ we have $\mathrm{K}_q(k') = \mathrm{N}_q(Z/k')$. For example, a field $k$ satisfies $C_i^0$ if and only if for every pair of integers $d,n \geq 1$, field extension $k'/k$, and hypersurface $Z$ in $\mathbf{P}^n_{k'}$ of degree $d$ such that $d^i \leq n$ the variety $Z$ admits a $0$-cycle of degree $1$. In the other extreme, a field $k$ satisfies $C_0^q$ if and only if for every tower $k''/k'/k$ of finite extensions the norm map $\mathrm{N}_{k''/k'}: \mathrm{K}_{q}(k'') \to \mathrm{K}_q(k')$ is surjective. \par
    
    As we already mentioned, the original motivation to consider the $C_i^q$ property was to give a Diophantine characterisation of the cohomological dimension of $k$ or more precisely, a slight modification of the cohomological dimension denoted by $\dim k$, see \cite[Definition~1]{KatoKuzumaki1986Dimension}. Explicitly, the following conjecture can be found in \cite{KatoKuzumaki1986Dimension}: \par
    \begin{conjecture} (\cite[Main Conjecture]{KatoKuzumaki1986Dimension})
     Let $i,q$ be non-negative integers. A field $k$ satisfies $C_i^q$ if and only if $\mathrm{dim} \, k \leq i+q$.    
    \end{conjecture}
    
    \noindent As it was pointed out in the original work \cite{KatoKuzumaki1986Dimension}, Bloch-Kato's conjecture (now resolved thanks to the work of Joukhovitsky, Rost, Suslin, Voevodsky among others) implies that Kato-Kuzumaki's conjecture holds when $i=0$. Unfortunately, counterexamples to this conjecture can be found in \cite[Proposition~2]{Merkurjev1991NotC20} and \cite{CTMadore2004NotC10} when $i > 0$. Nevertheless, these counterexamples are given by fields constructed in a very intricate manner. As such it is not unreasonable to expect that Kato-Kuzumaki's conjecture holds for fields usually appearing in arithmetic geometry, e.g. finitely generated fields over their prime subfield, local fields or separably closed fields. \par
    Kato-Kuzumaki's conjecture is not known to hold in many cases. Among the known examples we have the function field $\mathbf{C}(x_1,\hdots, x_n)$ as proved in \cite{Diego2018KK} and the Laurent series field $\mathbf{C}\lau{t_1}\lau{t_2}$ as proved in \cite{Wittenberg2015KKQp}. Even verifying the $C_i^q$ property for a given field can be very challenging for $i\neq 0$. For instance, the only known cases where $C_1^1$ property is verified are the following: $\mathbf{C}\lau{t_1}\lau{t_2}$, $p$-adic fields and totally imaginary number field due to \cite{Wittenberg2015KKQp} - a second proof in the case of totally imaginary number fields was given in \cite{Diego2018KK}, and $\mathbf{C}(x_1,x_2)$ due to \cite{Diego2018KK}. In the case of function fields very little is known when the base field is not algebraically closed — to the knowledge of the authors the only known results in this direction is when the function fields is one-dimensional and the base field is either a finite field (\cite{Diego2018KK}), a $p$-adic field (\cite{DiegoLuco2024KKp-adicFunction}), or an iterated Laurent series field over a $p$-adic field (\cite{Gambardella2025HigherLocalKK}).\par
    
    In the present manuscript we are mainly interested in the behaviour of the $C_i^q$ properties under field extensions —e.g. from a field $k$ satisfying $C_i^q$ to the field of rational functions $k(x)$. The main application of the methods of this paper is the following result:
    \begin{thmalpha}[Corollary \ref{cor Cn1 function finite}] \label{thmalpha function fields}
        The field $\mathbf{F}_p(x_1,\hdots, x_n)$ satisfies $C_n^1$.
    \end{thmalpha}

    \noindent There are two main ingredient in the proof of Theorem~\ref{thmalpha function fields}: the case $n=1$ that build upon the tools developed in \cite{Diego2018KK} and a transfer result in the spirit of \cite[Theorem 2a]{Nagata1957NotesPaperLang}. In order to state the transfer result, we need the following definition: For a class $\alpha \in \KK_q(k)$ we say that the field $k$ satisfies $C_i(\alpha)$ if for every hypersurface $Z \subseteq \P^n_k$ of degree $d$ such that $d^i\leq n$ we have $\alpha \in \mathrm{N}_q(Z/k)$. Moreover, a field $k$ is said to satisfy $C_i^q(k)$ if it satisfies $C_i(\alpha)$ for every $\alpha \in \KK_q(k)$. The main transfer result is the following:
    \begin{thmalpha}[Theorem \ref{thm constant transfer}]\label{thm into constant transfer}
        Let $k$ be a field of characteristic exponent $p$. Let $i,q \in \NN$ and $\alpha \in \KK_q(k)$. Assume that $k$ satisfies $C_i(\alpha)$ and $[k:k^p] \leq p^{i+1}$. Then, every finitely generated regular extension $K/k$ of finite transcendence degree $\delta$ satisfies $C_{i+\delta}(\alpha|_K)$.
    \end{thmalpha}
    \noindent This result can be used to get a similar transfer statement for Laurent series fields, in the same spirit to how the behaviour of the $C_i$ properties when one adds a formal variable (\cite[Theorem~2]{Greenberg1966RationalPointsHenselian}) can be deduced from its behaviour when one adds a rational variable (\cite[Theorem~2a]{Nagata1957NotesPaperLang}), see Corollary~\ref{cor constant transfer Laurent}. \par
    Lang and Nagata proved in \cite[Theorem 2a]{Nagata1957NotesPaperLang} that the $C_i$ property is stable under finite extensions. When trying to generalise this stability to the $C_i^q(k)$ property we have the additional difficulty of having to consider elements that lie in the Milnor $\KK$-groups of finite extensions of $k$, but may not come from Milnor $\KK$-theory of $k$ itself. In this direction, we obtain the following result.
    
    \begin{thmalpha}[Theorem \ref{thm: transferability under finite alg extension}]\label{thmalpha C}

Let $k$ be a field and $\alpha \in \KK_q(k)$. Assume that $k$ satisfying $C_i(\alpha)$ and if $k$ has positive characteristic $p$, further assume that $[k:k^p] \leq p^{i+1}$. Then, every finite Galois extension $l/k$ satisfies $C_i(\alpha|_l)$.

    \end{thmalpha}
    \noindent Theorem \ref{thm into constant transfer} has further applications in the context of function fields over various non-algebraically closed fields. Notably, we prove that the $C_i^q$ property for function fields is reduced to the $q$-dimensional case when the field is characteristic $0$. Explicitly, we get the following result.
    \begin{thmalpha}[Theorem \ref{thm transfer applications}]\label{thmalpha D}
        Let $i,q$ be a pair of non-negative integers and $k$ a field of exponent characteristic $p$ such that $k(x_1,\dots, x_q)^{\mathrm{perf}}$ satisfies $C_i^q$. Then, for all non-negative integers $j$ and $q' \leq q$, the field $k(x_1,\dots, x_{q'+j})$ satisfies the $C_{i+j}^{q'}$ property away from $p$.
    \end{thmalpha}

    \noindent When $k$ is perfect and $q=1$, we obtain a stronger result from which we deduced Theorem~\ref{thmalpha function fields}. Precisely, we have the following:
    \begin{thmalpha}[Theorem~\ref{thm transfer applications p-degree 1}]
        Let $i,q$ be a pair of non-negative integers and $k$ a perfect field such that $k(x)$ satisfies $C_i^1$. Then, for every non-negative integer $j$, the field $k(x_1,\dots, x_{j+1})$ satisfies the $C_{i+j}^{1}$~property.
    \end{thmalpha}
    \noindent In the same spirit, we obtain partial results over $p$-adic fields or totally imaginary number fields, see Section \ref{sec: Applications} for the precise results.

\subsection*{Outline of the paper}
\begin{itemize}
    \item In Section~\ref{sec: Preliminaries and notation} we fix the notation used throughout the remainder of the paper and we recall and prove of several results needed for later proofs.
    \item In Section~\ref{sec: Intersection of hypersurfaces} we study the norm groups of intersections of hypersurfaces over fields satisfying the $C^q_i$ property.
    \item In Section~\ref{sec: Stability under finite extensions} we use the results from Section~\ref{sec: Intersection of hypersurfaces} to prove Theorem~\ref{thmalpha C}.
    \item In Section~\ref{sec: Transfer results for C_i(alpha)} we prove Theorem~\ref{thm into constant transfer}. We then prove several transfer results for the $C_i(\alpha)$ property for Laurent series fields.
    \item In Section~\ref{sec: Applications} we discuss how to apply the results from the previous sections to obtain new examples of fields which satisfy the $C^q_i$ and $C_i(\alpha)$ properties. We begin by proving Theorem~\ref{thmalpha D} and we get Theorem~\ref{thmalpha function fields} as a direct application. We then discuss the $C_i(\alpha)$ property for function fields of geometrically integral varieties over $p$-adic fields and totally imaginary number fields.
    \end{itemize}
    \noindent In the appendices we discuss various cases the $C_i^q$ properties and Kato-Kuzumaki's conjecture that are not directly related with the tools developed in this manuscript, but that are useful for us and that are interesting in themselves.
    \begin{itemize}
    \item In Appendix \ref{Appendix Local fields} we prove that local fields in positive characteristic satisfy the $C_1^1$ property. This proof also works in characteristic zero, giving a new proof of the $C_1^1$ property for $\mathbf{Q}_p$.
    \item In Appendix \ref{appendix global function fields} we explain how the arguments in \cite{Diego2018KK} together with Appendix \ref{Appendix Local fields} are enough to prove that $\mathbf{F}_p(x)$ satisfies $C_1^1$ —not only \textit{away from} $p$.
    \item In Appendix \ref{Appendix KK separably closed field} we give a fairly elementary proof of Kato-Kuzumaki's conjecture for separably closed fields.
\end{itemize}

\section*{Acknowledgments}
    \noindent The first named author would like to express his deep gratitude towards Diego Izquierdo for his support and guidance. The second named author would also like to thank Daniel Loughran for his tireless support throughout the process of writing the paper. We would also like to thank Diego Izquierdo, Daniel Loughran, and Olivier Wittenberg for their comments on the structure of the paper. We are also thankful to Kai Huang for pointing out a mistake in the original version of the manuscript.

\section{Preliminaries and notation}
\label{sec: Preliminaries and notation}
In this section we fix the notation we will use throughout this paper and discuss some important preliminary results.\par
\subsection*{Milnor \texorpdfstring{$\KK$}{K}-theory}

Let $k$ be a field and $q$ a non-negative integer. We define the $q$-th Milnor $\mathrm{K}$-group of $k$, denoted $\mathrm{K}_q(k)$, as
 \begin{equation*}
  \mathrm{K}_q(k) :=\begin{cases}
      \ZZ, &\text{ if }q=0,\\
       \left(k^{\times} \right)^{\otimes q} \; \big/ \; \left\langle a_1 \otimes \dots \otimes a_q \, | \, \exists i \neq j,\; a_i + a_j = 1 \right\rangle &\text{ otherwise,}
  \end{cases}   
 \end{equation*}
where the tensor product is taken over $\ZZ$. \par
    
For $a_1, \hdots, a_q \in k^{\times}$ we denote by $\{a_1, \hdots, a_q\}$ the class of $a_1 \otimes \hdots \otimes a_q$ in $\mathrm{K}_q(k)$. Elements of this form are called \textit{symbols}. More generally, for any pair of non-negative integers $p,q$ there is a natural pairing
    \[\{\cdot, \cdot\}: \mathrm{K}_p(k) \times \mathrm{K}_q(k) \to \mathrm{K}_{p+q}(k)\]
    induced by the tensor pairing $\left(k^{\times} \right)^{\otimes p} \times \left(k^{\times} \right)^{\otimes q} \to \left(k^{\times} \right)^{\otimes (p+q)}$. \par

Let $l/k$ be a finite extension. One can construct a norm homomorphism
    \[\mathrm{N}_{l/k}: \mathrm{K}_q(l) \to \mathrm{K}_q(k)\]
    satisfying the following properties:
    \begin{itemize}
        \item for $q = 0$, the map $\mathrm{N}_{l/k}: \KK_0(l)\to \KK_0(k)$ is given by multiplication by $[l:k]$,
        \item for $q = 1$ the map $\mathrm{N}_{l/k}: \mathrm{K}_1(l) \to \mathrm{K}_1(k)$ coincides with the usual norm $l^{\times} \to k^{\times}$,
        \item for any pair of non-negative integers $p,q$, we have $\mathrm{N}_{l/k} (\{\alpha, \beta\}) = \{\alpha, \mathrm{N}_{l/k}(\beta) \}$ for $\alpha \in \mathrm{K}_p(k)$  and $\beta \in \mathrm{K}_q(l)$, and
        \item if $m$ is a finite extension of $l$ we have $\mathrm{N}_{m/k} = \mathrm{N}_{l/k} \circ \mathrm{N}_{m/l}$.
        \item The composition $\mathrm{N}_{l/k} \circ i_{l/k}$ is given by $\alpha \mapsto \alpha^{[l:k]}$ where $i_{l/k}: \mathrm{K}_q(k) \to \mathrm{K}_q(l)$ is the natural morphism induced by the inclusion $k \subseteq l$.
    \end{itemize}
    The construction can be found in \cite[Section 1.7]{Kato1980GeneralizationClassFieldK2} or \cite[Section 7.3]{GilleSzamuely2017CentralGalois}. Usually, we denote $i_{l/k}(\alpha)$ by $\alpha|_{l}$.

\subsection*{Kato-Kuzumaki's conjecture}
    Let $Z$ be a $k$-scheme of finite type. For any non-negative integer $q$, we define the $q$\textit{-th norm group} of $Z$ as 
    \[\mathrm{N}_q(Z/k) = \left\langle \mathrm{N}_{k(x)/k}(\mathrm{K}_q(k(x))) \; | \; x \in Z_{(0)}\right\rangle \subseteq \mathrm{K}_q(k)\]
    where $Z_{(0)}$ denotes the set of closed points of $Z$ and $k(x)$ denotes the residue field of $Z$ at $x$. When $l/k$ is a finite extension we write $\mathrm{N}_q(l/k)$ for $\mathrm{N}_{q}(\Spec (l) /k)$. \par
    \begin{defi}
    The field $k$ is said to satisfy the \textit{$C_i^q$ property} if for every pair of non-negative integers $n,d$, finite extension $l/k$, and hypersurface $Z$ of $\mathbf{P}^n_l$ of degree $d$ with $d^i \leq n$ we have $\mathrm{N}_q(Z/l) = \mathrm{K}_q(l)$. 
    \end{defi}
    \noindent As it was already pointed out in the original article \cite{KatoKuzumaki1986Dimension}, the Bloch-Kato conjecture (now the norm-residue isomorphism) implies that the $C_0^q$ property is equivalent to having cohomological dimension at most $q$ (as defined in \cite[Definition~1]{KatoKuzumaki1986Dimension}). \par
    We say that a field $k$ satisfies the \textit{Kato-Kuzumaki conjecture} if $\dim k =\delta$ is finite (as defined in \cite[Definition 1]{KatoKuzumaki1986Dimension}) and satisfies $C_i^q$ for every pair of non-negative integers $i,q$ such that $\delta \leq i+q$. \par
    The following lemma corresponds to \cite[Lemma 1]{KatoKuzumaki1986Dimension}, which is stated without proof. We include a proof for completeness. We denote by $k(\ell)/k$ the maximal prime-to-$\ell$ extension i.e. the composite extension of all finite extensions of $k$ of degree prime to $\ell$.
    \begin{lemma} \label{lemma Ci0 prime to l}
        A field $k$ satisfies $C_i^0$ if and only if for every prime $\ell$, the field $k(\ell)$ is $C_i$.
    \end{lemma}
    \begin{proof}
        Assume that $k$ satisfies $C_i^0$. Let $\ell$ be a prime number. For every hypersurface $Z \subseteq \mathbf{P}^n_{k(\ell)}$ there exists a finite extension $l/k$ such that $Z$ is defined over $l$ and $l\subset k(\ell)$. If the degree $d$ of $Z$ satisfies $d^i\leq n$ we can find extensions $l_1,\hdots , l_m/l$ such that $Z(l_j)\neq \emptyset$ and $\mathrm{gcd}([l_1:l],\hdots , [l_m:l]) = 1$. In particular, there exists an index $j$ such that $[l_j:l]$ is coprime with $\ell$. By construction of $k(\ell)$ this ensures that $l_j \subseteq k(\ell)$ and thus, $Z(k(\ell)) \neq \emptyset$. \par
        Conversely, assume that for every prime number $\ell$ the field $k(\ell)$ is $C_i$. Let $l/k$ be a finite extension and $Z\subseteq \mathbf{P}^n_l$ a hypersurface of degree $d$ such that $d^i \leq n$. Let $m/l$ be a finite extension such that $Z(m)\neq \emptyset$. Since $k(\ell)$ is $C_i$ for every prime number $\ell$, for any prime number $\ell'$ dividing $[m:l]$ we can find a finite subextension $m_{\ell'}$ of $l(\ell')/l$ such that $Z(m_{\ell'})\neq\emptyset$. In particular, since $[m_{\ell'}:l]$ is prime to $\ell'$, we have that $[m:l]$ and $\left\{[m_{\ell}:l] \; | \; \ell \text{ divides } [m:l] \right\}$ are coprime setwise.
    \end{proof}
    \noindent More generally, Lemma~\ref{lemma Ci0 prime to l} generalises to the $C_i^q$ properties with $q \neq 0$. A proof can be found in \cite[Lemma 2.5]{DiegoLuco2023TransferSerreII}.

\subsection*{Milnor \texorpdfstring{$\KK$}{K}-theory of Henselian fields}
    Let $K$ be a Henselian discrete valuation field with ring of integers $R$, maximal ideal $\mathfrak{m}$ and residue field $k$. Then, for every strictly positive integer $q$ there exists a unique morphism
    \[ \partial : \mathrm{K}_q(K) \to \mathrm{K}_{q-1}(k)\]
    such that for every uniformiser $\pi$ of $R$ and units $u_2,\hdots u_q \in R^{\times}$ we have
    \[\partial(\{\pi, u_2, \hdots , u_q \}) = \{\overline{u_2}, \hdots , \overline{u_q}\}\]
    where $\overline{u_2}, \hdots , \overline{u_q}$ denotes the images of $u_2,\hdots u_q$ in the residue field. We call this map the \textit{residue map}. \par
    We denote the kernel of $\partial: \mathrm{K}_q(K) \to \mathrm{K}_{q-1}(k)$ by $\mathrm{U}_q(K)$. This group is generated by symbols of the form $\{u_1 , \hdots ,u_q \}$ where $u_1, \hdots ,u_q$ are units in $R$, see \cite[Proposition~7.1.7]{GilleSzamuely2017CentralGalois}. One can also define a specialisation map $s: \mathrm{U}_q(K) \to \mathrm{K}_{q}(k)$ characterised by $s(\{u_1 , \hdots ,u_q \}) = \{\overline{u_1}, \hdots , \overline{u_q}\}$, see \cite[Proposition 7.1.4]{GilleSzamuely2017CentralGalois}. Denote by $\mathrm{U}_q^1(K)$ the kernel of $s$. It is generated by symbols of the form $\{u, a_2, \hdots, a_q\}$ where $u$ belongs to $1 + \mathfrak{m}$ and $a_2,\hdots, a_q\in K^{\times}$, see \cite[Proposition~7.1.7]{GilleSzamuely2017CentralGalois}. \par
    Moreover, the residue map and specialisation map are compatible with the norm map. Indeed, if $L/K$ is a finite extension with ramification index $e$ and the residue field of $L$ is $l$, we have the following commutative diagrams whose horizontal arrows are isomorphisms
    \begin{eqnarray*}
        \begin{tikzcd}
            \mathrm{K}_q(L)/ \mathrm{U}_q(L) \ar[r, "\sim" ', "\partial_L"]  \ar[d,"\mathrm{N}_{L/K}"]& \mathrm{K}_{q-1}(l) \ar[d,"\mathrm{N}_{l/k}"] & \mathrm{U}_q(L)/ \mathrm{U}_q^1(L) \ar[r, "\sim" ', "s_L"]  \ar[d,"\mathrm{N}_{L/K}"]& \mathrm{K}_{q}(l) \ar[d,"e \cdot \mathrm{N}_{l/k}"] \\
            \mathrm{K}_q(K)/ \mathrm{U}_q(K) \ar[r, "\sim" ', "\partial_K"] & \mathrm{K}_{q-1}(k)  & \mathrm{U}_q(K)/ \mathrm{U}_q^1(K) \ar[r, "\sim" ', "s_K"]  & \mathrm{K}_{q}(k).
        \end{tikzcd}
    \end{eqnarray*}
For every prime number $\ell$ different from the characteristic of $k$, the group $\mathrm{U}_q^1(K)$ is $\ell$-divisible. Furthermore, in the equicharacteristic case, when $L/K$ is an unramified separable extension we have the following:
\begin{lemma}\label{lem: Z has K^sep point implies U^1(K) norm of Z}
Let $l/k$ be a finite separable extension and $q\in\ZZ_{\geq 1}$. Fix $K:= k\lau{t}$ and $L:=l\lau{t}$ then $\mathrm{U}_q^1(K)\subseteq\mathrm{N}_q(L/K)$.
\end{lemma}

\begin{proof}
By \cite[Proposition 7.1.7]{GilleSzamuely2017CentralGalois} the group $\mathrm{U}_q^1(K)$ is generated by symbols of the form $\{x_1,\hdots , x_n\}$, where $x_1\in 1+t k[\![t]\!]$. Thus it suffices to prove that $1+t k[\![t]\!]$ is contained in $\mathrm{N}_{L/K}(L^{\times})$. Let $\alpha=1+\sum^\infty_{i=1}a_it^i$ be an element of $1+t k[\![t]\!]$. We produce by induction a sequence of elements $\{\beta_n\}_{n\in\NN}\subset k[\![t]\!]$ satisfying
\begin{align*}
    \mathrm{N}_{L/K}(\beta_n) & \equiv \alpha \;\, \mod t^{n+1},\\
    \beta_{n+1}& \equiv \beta_{n}  \mod t^{n+1}.
\end{align*}
In particular by completeness we will obtain an element $\beta\in L$ such that $\operatorname{N}_{L/K}(\beta)= \alpha$. We set $\beta_0= 1$. Now assume $\beta_n$ has been constructed for all $n<N$. By construction we have
\begin{equation*}
    \frac{\alpha}{\mathrm{N}_{L/K}(\beta_{N-1})}\equiv 1+b_{N}t^N\mod t^{N+1},
\end{equation*}
for some $b_N\in k$ (since $\frac{\alpha}{\operatorname{N}_{L/K}(\beta_{N-1})} \equiv 1 \mod t^N$). Now since $l/k$ is separable, the trace map is surjective and hence there exists $c_n\in l$ such that $\operatorname{tr}(c_N)=b_N$. Thus we set 
\begin{equation*}
    \beta_N:=\beta_{N-1}\left(1+c_Nt^N\right).
\end{equation*}
Clearly $\beta_N\equiv\beta_{N-1}\mod t^N$, and by definition we have
\begin{equation*}
   \operatorname{N}_{L/K}\left(1+c_Nt^N\right)\equiv 1+\sum_{\sigma: L \hookrightarrow \overline{K}}\sigma(c_N)t^N\equiv 1+\operatorname{tr}(c_N)t^N\equiv1+b_Nt^N\mod t^{N+1}, 
\end{equation*}
where the sum is over all $K$-embeddings of $L$ in an algebraic closure of $K$. Therefore we have
\begin{equation*}
    \frac{\alpha}{\operatorname{N}_{L/K}(\beta_N)} \equiv \frac{\alpha}{\operatorname{N}_{L/K}(\beta_{N-1})\operatorname{N}_{L/K}\left(1+c_Nt^N\right)}\equiv 1\mod t^{N+1},
\end{equation*}
and hence by the principle of induction the result follows.
\end{proof}

\noindent In particular, if $K$ is equicharacteristic and a variety $Z$ over $K$ has a point over $k^sK$ where $k^s$ is the separable closure of $k$, then we have $\mathrm{U}_q^1(K)\subseteq\operatorname{N}_q(Z/K)$. 

\begin{rmk}
    Lemma \ref{lem: Z has K^sep point implies U^1(K) norm of Z} fails if $L/K$ is replaced by an Artin-Schreider extension. Indeed, set $K_0$ to be a finite field and $L$ the extension $K(\alpha)$ where $\alpha$ is a root of $X^p - X - t^{-1}$. From \cite[XIV~\S~5 Corollary to Proposition~15]{Serre1979LocalFields} one can deduce that an element of the form $1+ c t$ is a norm of $L/K$ if and only if $\operatorname{Tr}_{K_0/\mathbf{F}_p}(c) = 0$.
\end{rmk}
\subsection*{Milnor \texorpdfstring{$\KK$}{K}-theory of rational function fields}  \label{sec Milnor rational functions} 

Let $K := k(x)$ be the field of rational functions over a field $k$. Every closed point $p$ of the affine line $\mathbf{A}_k^1$ corresponds to a discrete valuation over $K$. Let $k(p)$ be the residue field of $p$ and $K_p$ the completion of $K$ with respect to that valuation. We denote by $\partial_p: \mathrm{K}_{q}(K) \to \mathrm{K}_{q-1}(k(p))$ the composition of the natural map induced by the inclusion $K \hookrightarrow K_p$ and the residue map described above. Then, for every $q \geq1$ there is an exact sequence 
    \[0 \to \mathrm{K}_{q}(k) \to \mathrm{K}_{q}(K) \xrightarrow{(\partial_p)} \bigoplus_{p \in \mathbf{A}^1_{k,(0)}} \mathrm{K}_{q-1}(k(p)) \to 0,\]
    see \cite[Theorem 7.2.1]{GilleSzamuely2017CentralGalois}. We refer to this exact sequence as \textit{Milnor's exact sequence}. \par
    We also note that the map $(\partial_p)$ is compatible with finite extensions of $k$. Indeed, for every finite extension $l/k$ we have the following commutative diagram
    \begin{equation*}
        \begin{tikzcd}
            \mathrm{K}_{q}(lK) \ar[r] \ar[d, "\mathrm{N}_{lK/K}"] & \bigoplus_{p' \in \mathbf{A}^1_{l,(0)}} \mathrm{K}_{q-1}(l(p')) \ar[d,"\prod_{p'}\mathrm{N}_{p'/p}"] \\
            \mathrm{K}_{q}(K) \ar[r]  & \bigoplus_{p \in \mathbf{A}^1_{k,(0)}} \mathrm{K}_{q-1}(k(p)).
        \end{tikzcd}
    \end{equation*}
    The $p$-coordinate of $\prod_{p'}\mathrm{N}_{p'/p}((\alpha_{p'})_{p'})$ is given by $\prod_{p'\in \pi^{-1}(p)} \mathrm{N}_{l(p')/k(p)}(\alpha_{p'})$, where $\pi:~\mathbf{A}^1_l \to \mathbf{A}^1_k$ is the natural projection, see \cite[Corollary 7.4.3]{GilleSzamuely2017CentralGalois}.

\section{Intersection of hypersurfaces}
\label{sec: Intersection of hypersurfaces}
    In this section we discuss the behaviour of intersections of hypersurfaces over fields satisfying $C_i^q$. From now on, we fix an algebraic closure $\overline{k}$ of the field $k$ and every finite extension of $k$ will be considered as a subextension of $\overline{k}/k$.

    \subsection{Lang's construction and \texorpdfstring{$l$}{l}-normic forms}
    \label{subsec: Lang's construction}
    \noindent In this subsection we recall a construction due to Lang in \cite{Lang1952QuasiAlgebraicClosure}. \par
    Let $N,n$ and $r$ be integers such that $N-1 \geq r \geq 1$. Let $\Phi \in k[X_0, \hdots , X_N]$ be a homogeneous polynomial of degree $e$, and $F_1, \hdots, F_r \in k[X_0, \hdots , X_n]$ homogeneous polynomials of degree $d$. For every integer $\mu \geq 2$ we define inductively a homogeneous polynomial $\Phi^{(\mu)} \in k[X_0, \hdots ,X_{N_{\mu}-1}]$ of degree $D_{\mu}$, where $D_{\mu} =  e d^{\mu - 1}$ and $N_{\mu} = (n+1)\left\lfloor \frac{N_{\mu -1}}{r}\right\rfloor$. The base case is defined by $\Phi^{(1)} = \Phi$ and $N_1 = N+1$. Denote by $\underline{F}(X_0,\hdots,X_n)$ the vector $(F_1(X_0,\hdots , X_n),\hdots ,F_r(X_0, \hdots, X_{n}))$ and similarly for other sets of variables. In general, we define $\Phi^{(\mu)}$ as  
        \begin{multline*}
            \Phi^{(\mu)}(X_0, \hdots, X_{N_{\mu}-1}) = \Phi^{(\mu-1)}(\underline{F}(X_0,\hdots , X_n),\underline{F}(X_{n+1}, \hdots, X_{2n+1}),\hdots, \\  \hdots, \underline{F}(X_{N_{\mu}-n -1},\hdots X_{N_{\mu}-1}), 0, \hdots, 0).
        \end{multline*}
    Let $Z_1, \hdots Z_r\subseteq \P^n_k$ be the hypersurfaces defined by $F_1, \hdots, F_r$ respectively. We denote by $\Phi^{(\mu)}(Z_1, \hdots , Z_{r})$ the hypersurface of $\P^{N_{\mu}-1}_k$ defined by $\Phi^{(\mu)}$ and by $Z_{\Phi}$ the hypersurface of $\P^{N}_k$ defined by $\Phi = 0$. The main property of this construction is the following: if the intersection $Z \coloneq \bigcap_{i=1}^r Z_i$ and $Z_{\Phi}$ have no rational point, the hypersurface $\Phi^{(\mu)}(Z_1, \hdots , Z_{r})$ has no rational point either. The following proposition is just a rephrasing of this property.
        \begin{prop}\label{prop higher Lang's construction}
            Keep the previous notation. Let $l/k$ be a finite extension and $\mu \in~\mathbf{N}$. If $\Phi^{(\mu)}(Z_1,\hdots ,Z_r)$ admits an $l$-point, then either $Z(l)$ is non-empty or $Z_{\Phi}(l)$ is non–empty. 
        \end{prop}
        \noindent Another elementary yet relevant property of this construction is the following:
        \begin{prop}\label{prop Langs construction tending to infty}
            Keep the previous notation. Let $i \geq 1$ be an integer. Suppose that $r d^i \leq n$. Then the quantity $N_{\mu} / D_{\mu}^i$ goes to infinity as $\mu$ does.
        \end{prop}
    \noindent One of the key ideas of the proof of \cite[Theorem 1a]{Nagata1957NotesPaperLang} is to use the previous construction with $\Phi$ a \textit{normic form}. In our context, normic forms in the sense of \cite{Lang1952QuasiAlgebraicClosure} are not enough to prove the main result of the section. 
    \begin{defi}\label{def l-normic forms}
        Let $l/k$ be a finite extension. We say that a hypersurface $H \subseteq \P_k^n$ is an $l$-normic form if it admits a point in a finite extension $k'/k$ if and only if $l\subseteq k'$.
    \end{defi}
    \noindent Note that the previous definition only makes sense because we fixed an algebraic closure $\overline{k}$ of $k$ and we consider all finite extension as subextensions of $\overline{k}$.
    \begin{prop}\label{prop existence l-normic}
        Let $l/k$ be a finite normal simple extension. Then there exists an $l$-normic form of arbitrarily large dimension. 
    \end{prop}
    \begin{proof}
        Let $a\in l^{\times}$ be a primitive element of $l$ and $p_a(T) \in k[T]$ its minimal polynomial over $k$. Consider the homogenisation $F(X_0,X_1)$ of $p_a$. Since $l/k$ is normal, the closed point $Z$ of $\P^1_k$ defined by $F(X_0,X_1)=0$ is an $l$-normic form. Consider the hypersurfaces $\Phi^{(\mu)}(Z)$ obtained by applying Lang's construction with $r=1$ and $\Phi = F = F_1$. By Proposition \ref{prop higher Lang's construction} we have that $\Phi^{(\mu)}(Z)$ only admits points in extensions containing $l/k$. Moreover, since $F$ has $2$ variables, we have $N_{\mu} = 2^{\mu}$, hence the dimension of $\Phi^{(\mu)}(Z)$ can be made arbitrarily large.
    \end{proof}
    
    \subsection{From hypersurfaces to intersections}
    \noindent An important step in the proof of \cite[Theorem 2a]{Nagata1957NotesPaperLang} is considering intersections of hypersurfaces. The following results are the analogue of \cite[Theorem 1a]{Nagata1957NotesPaperLang}.
    \begin{prop}\label{prop intersection hypersurfaces}
        Let $q,i$ be non-negative integers and $\alpha \in \mathrm{K}_q(k)$. Let $Z_1, \hdots, Z_r \subseteq \P^n_k$ hypersurfaces of degree $d$ such that $r d^i \leq n$. Assume that $Z\coloneq \bigcap_{i=1}^r Z_i$ admits a point on a simple normal extension of $k$ and that $k$ satisfies $C_i(\alpha)$. Then, we have $\alpha \in \mathrm{N}_q(Z/k)$.
    \end{prop}
    \begin{proof}
        Consider a simple normal extension $l/k$ such that $Z(l)\neq \emptyset$.  Applying Proposition \ref{prop existence l-normic} to the extension $l/k$ we obtain a hypersurface $H =\{\Phi = 0\} \subseteq \P^n_{k}$ with $r \leq n$ that only admits points on extensions of $l$. Since $k$ satisfies $C_i(\alpha)$, Proposition~\ref{prop Langs construction tending to infty} ensures the existence of $\mu\in\NN$ such that $\alpha \in \mathrm{N}_{q}(\Phi^{(\mu)}(Z_1,\hdots ,Z_r)/k)$. Explicitly, this means that there exist finite extensions $l_1,\hdots ,l_s/k$ and classes $\beta_j \in \mathrm{K}_q(l_j)$, for each $1\leq j\leq s$, such that $\Phi^{(\mu)}(Z_1,\hdots , Z_r)(l_j) \neq \emptyset$ and $\alpha = \prod_{j=1}^s \mathrm{N}_{l_j/k}(\beta_j)$. By Proposition~\ref{prop higher Lang's construction} we deduce that for every $1\leq j \leq s$ we either have $Z(l_j) \neq \emptyset$ or $H(l_j)\neq \emptyset$. Since $H$ is an $l$-normic form we have $l \subseteq l_j$ for every $1\leq j\leq s$ such that $H(l_j)\neq \emptyset$. Thus, since $Z$ admits a point over $l$ we have that $Z(l_j)\neq\varnothing$ every $1\leq j\leq s$. Therefore for every $1\leq j\leq s$ the set $Z(l_j)$ is non-empty, and hence we have that $\alpha \in \mathrm{N}_q(Z/k)$.
    \end{proof}
\begin{rmk}\label{remark: failure of norm forms}
        Let $l/k$ be a finite Galois extension of prime degree or, more generally, that it does not have any intermediate non-trivial extensions. Fix a primitive element $\omega \in l$. Then, the hypersurface given by 
        \begin{equation*}
            \mathrm{N}_{l/k}(X_0 + X_1 \omega + \hdots , X_{[l:k]-1} \omega^{[l:k]-1})=0,
        \end{equation*}
         is an $l$-normic form. These are the normic forms used in \cite{Nagata1957NotesPaperLang}, but they are not enough to prove Proposition~\ref{prop intersection hypersurfaces} since the intersection of hypersurfaces does not necessarily have a point in such an extension.
    \end{rmk}
    
    \noindent Thanks to the primitive element theorem, Proposition \ref{prop intersection hypersurfaces} applies to any perfect field. It can also apply for imperfect fields when $[k:k^p] \leq p^i$ due to the following folklore~proposition.
    \begin{prop}\label{prop separable closed Cdelta}
        Let $k$ be a separably closed field of positive characteristic $p$. Assume that $[k:k^p]\leq p ^{\delta}$. Then the field $k$ is $C_{\delta}$.
    \end{prop}
    \begin{proof}
        If $\delta'$ is such that $[k:k^p]=p^{\delta'}$, the field $k$ is elementary equivalent in the sense of model theory to the separable closure of $\mathbf{F}_p(x_1,\dots, x_{\delta'})$ thanks to \cite[Proposition]{Ershov1067FieldsSolvableTheory}, and this is a $C_{\delta'}$ field due to \cite[Theorem 2a]{Nagata1957NotesPaperLang}. Since the $C_i$ property can be stated in first order logic, we conclude that $k$ is $C_{\delta'}$, and hence $C_{\delta}$. 
    \end{proof}
    \noindent Furthermore, since we only require these intersections to have a point over a simple normal extension in order to apply proposition~\ref{prop intersection hypersurfaces} and not a Galois extension, we can slightly improve this bound.
    \begin{lemma}\label{lemma p-degree Ci}
        Let $k$ be a field of positive characteristic $p$ and $\delta \in \mathbf{N}$ such that $[k:k^p] \leq p^{\delta}$. Let $Z_1,\dots , Z_r \subseteq \mathbf{P}^n_k$ be hypersurfaces of degree $d$ such that $rd^{\delta -1} \leq n$. Denote by $Z$ the intersection $\bigcap_{j=1}^rZ_j$. Then, there exists a simple normal extension $l/k$ such that $Z(l) \neq \emptyset$. 
    \end{lemma}
    \begin{proof}
        Denote by $k^s$ the separable closure of $k$. Now fix $x \in k \setminus k^p$. The field $K := \bigcup_{m\geq 1} k^s(x^{1/p^m})$ satisfies $[K:K^p]\leq p^{\delta-1}$ due to \cite[Lemma~3.15]{DiegoLuco2023TransferSerreII}. We can apply Proposition \ref{prop separable closed Cdelta} and \cite[Theorem 1a]{Nagata1957NotesPaperLang} to ensure that $Z(K) \neq \emptyset$. In particular, there exists a finite subextension $l/k$ of $K/k$ such that $Z(l)\neq \emptyset$. We may assume that $l$ is normal since its normal closure is contained in $K$. It remains to prove that $l$ is simple. By construction there exists a separable extension $l_s/k$ and a $m \in \mathbf{N}$ such that $l=l_s(x^{1/p^m})$. This implies that every subextension $k'/k$ of $l/k$ is a separable subextension of $l/k(x^{1/p^{m'}})$ for some $m' \leq m$. We deduce that $l/k$ has only finitely many subextensions because for every $m'\leq m$ the extension $l/k(x^{1/p^{m'}})$ has only finitely many separable subextensions. Hence, $l/k$ is a simple extension.
    \end{proof}
    \noindent From this lemma and Proposition \ref{prop intersection hypersurfaces} we can deduce the following analogue of \cite[Theorem~1a]{Nagata1957NotesPaperLang}.
    \begin{cor}
        Let $k$ be a field satisfying $C_i^q$. If $k$ has positive characteristic $p$, further assume that $[k:k^p]\leq p^{i+1}$. Then, the following condition is satisfied for all finite extensions $l/k$: Let $Z_1,\hdots ,Z_m \subseteq \P^n_l$ be hypersurfaces of degree $d$ such that $md^i\leq n$ and $Z$ the intersection $\bigcap_{j=1}^mZ_j$. Then, we have $\KK_q(l) = \mathrm{N}_q(Z/l)$.
    \end{cor}

\section{Stability under finite extensions}
\label{sec: Stability under finite extensions}
In this section we study the stability of the $C_i^q$ property under finite algebraic extensions. To do this, we consider the following variant of the $C_i^q$ property.
\begin{defi}
Let $k$ be a field and $i,q$ a pair of non-negative integers. We say $k$ satisfies the $C^q_i(k)$ property if for every hypersurface $Z\subseteq\P^n_k$ of degree $d$ such that $d^i\leq n$ we~have
\begin{equation*}
    \operatorname{N}_q(Z/k)=\operatorname{K}_q(k).
\end{equation*}    
\end{defi}
\noindent In particular, every finite extension $l/k$ satisfies $C_i^q(l)$ if and only if $k$ satisfies $C_i^q$.\par
The aim of this section is to prove the following result.

\begin{thm}
\label{thm: transferability under finite alg extension}
Let $k$ be a field and $\alpha \in \KK_q(k)$. Assume that $k$ satisfying $C_i(\alpha)$ and if $k$ has positive characteristic $p$, further assume that $[k:k^p] \leq p^{i+1}$. Then, every finite Galois extension $l/k$ satisfies $C_i(\alpha|_l)$.
\end{thm}

 \noindent Before proving, let us verify the following result regarding Weil restrictions.
    \begin{lemma}\label{lemma hypersurface and finite extensions}
        Let $l/k$ be a finite Galois extension of degree $m$ and $H \subseteq \P^n_l$ a hypersurface of degree $d$. Then, there exist hypersurfaces $Z_1,\hdots , Z_{m} \subseteq \P^{(n+1)m -1}_k$ of degree $d$ such that for every finite extension $k'/k$ the intersection $Z \coloneq  \bigcap_{j=1}^m Z_j$ admits a $k'$-point if and only if we have $H(l\otimes_k k') \neq \emptyset$.
    \end{lemma}
    \begin{proof}
        Let $F(X_0,\hdots,X_n)$ be the polynomial defining $H$. Fix a basis $ \omega_1,\dots, \omega_{m}$ of $l$ over $k$. We can replace the variable $X_i$ by $\underline{T}_i =T_{i,1}\omega_1 + \hdots + T_{i,m}\omega_{m}$. After rearranging the equation and expanding the constants, we obtain the decomposition
        \[
            F(\underline{T}_0,\hdots, \underline{T}_n)= \sum_{j=1}^{m}F_{j}(T_{0,1}, \hdots , T_{n,m}) \omega_{j},
        \]
        for certain homogeneous polynomials $F_j$ of degree $d$ with coefficients in $k$. For every $1\leq j \leq m$ we let $Z_j$ be the hypersurface in $\P^{(n+1)m-1}_k$ defined by $F_j =0$. It is clear that the intersection $Z$ has a point over $k'$ if $H(l\otimes_k k')\neq \emptyset$. \par 
        Conversely, let us consider $[t_{0,1}: \dots : t_{n,m}] \in Z(k')$. The tensor product $l \otimes_k k'$ decomposes as a product $\prod_{i=1}^s k_i$ of field $k_i$. One can verify that $A\coloneq(\sum_{j=1}^{m} t_{0,j}\otimes \omega_j , \dots , \sum_{j=1}^{m} t_{n,j}\otimes \omega_{j})$ is a non-zero solution of $F = 0$ in $l \otimes_k k'$. Unfortunately, it is not enough to determine an $l \otimes_k k'$-point of $H$. However, under the identification $l \otimes_k k' = \prod_{i=1}^s k_i$, there exists an index $i_0$ such that the $i_0$-th component of $A$ is non-zero —in particular, we have $H(k_i) \neq \emptyset$. Since the extension $l/k$ is Galois, the fields $k_i$ are all isomorphic over $k$. Noting $H(l\otimes_k k') = \prod_{i=1}^sH(k_i)$, we deduce that $H(l\otimes_k k')$ is non-empty.
    \end{proof}

    \noindent It should be noted that the variety $Z$ in Lemma~\ref{lemma hypersurface and finite extensions} is not isomorphic to Weil restriction of $H$. However, the Weil restriction of $H$ is the intersection of the equations $F_j$ defining $Z$ inside the Weil restriction $\mathrm{R}_{l/k}(\mathbf{P}^n_l)$, but not when embedded in $\mathbf{P}_k^{(n+1)m - 1}$. Nevertheless, a Segre-like embedding realises $\mathrm{R}_{l/k}(\mathbf{P}^n_l)$ as a closed subvariety of $\mathbf{P}_k^{(n+1)m - 1}$ and under this embedding we have
    \[
        \mathrm{R}_{l/k}(H) \simeq Z \cap \mathrm{R}_{l/k}(\P^n_l) \subseteq \mathbf{P}_k^{(n+1)m - 1}.
    \]
    So, Lemma~\ref{lemma hypersurface and finite extensions} states that in the case of Galois extensions, $Z(k')\neq \emptyset$ implies that $\mathrm{R}_{l/k}(Z)(k')$ is non-empty.

\begin{proof}[Proof of Theorem~\ref{thm: transferability under finite alg extension}]
    Let $H\subseteq \P^n_l$ be a hypersurface of degree $d$ such that $d^i\leq n$. Let $Z_1,\dots , Z_m\subseteq \P^{m(n+1)-1}_k$ be the hypersurfaces given by Lemma~\ref{lemma hypersurface and finite extensions} and denote by $Z \coloneq  \bigcap_{j=1}^{m} Z_j$ their intersection. By Proposition \ref{prop intersection hypersurfaces} and Lemma \ref{lemma p-degree Ci}, for every $\alpha \in \KK_q(k)$ there exist finite extensions $k_1, \hdots , k_r/k$ and $\beta_j \in \KK_q(k_j)$ such that $Z(k_j)\neq \emptyset$ and $\alpha~=~\prod_{j=1}^r\mathrm{N}_{k_j/k}(\beta_j)$. \par
    Since the extension $l/k$ is Galois, the tensor product $k_j \otimes_k l$ can be decomposed as a product of field $n_{j,s}$. Then, \cite[Lemma 7.3.6]{GilleSzamuely2017CentralGalois} ensures that the following diagram is~commutative.
    \begin{equation*}
        \begin{tikzcd}[column sep = 10ex, row sep = 6ex]
            \prod_{j=1}^r \KK_q(k_j) \ar[r] \ar[d,swap,"\prod i_{n_{j,s}/k_j}"] & \KK_q(k) \ar[d, "i_{l/k}"]\\
             \prod_{j=1}^r \prod_{s=1} \KK_q(n_{j,s}) \ar[r,"\prod \mathrm{N}_{n_{j,s}/l}"]& \KK_q(l). 
        \end{tikzcd}
    \end{equation*}
    Thie is enough to conclude becuase Lemma~\ref{lemma hypersurface and finite extensions} ensures that $H(n_{j,s}) \neq \emptyset$.
\end{proof}

\section{Transfer results for \texorpdfstring{$C_i(\alpha)$}{Cialpha}}
\label{sec: Transfer results for C_i(alpha)}
    \noindent In this section we prove the main result of the paper, transfer results for the $C_i(\alpha)$ property from $k$ to function fields over $k$ and to the Laurent series fields. These results are of different nature depending on whether $\alpha$ is a constant class or not. 
   \subsection{Transfer principle for constant classes}
    \noindent Let $K/k$ be an extension. We call a class $\alpha \in \mathrm{K}_q(K)$ \textit{constant} if it belongs to the image of the restriction $\KK_q(k) \to \KK_q(K)$. In the case where $K$ is a function field over $k$ and $q=1$ constant classes are indeed constant functions.
    \begin{thm}\label{thm constant transfer}
        Let $k$ be a field of characteristic exponent $p$. Let $i,q \in \NN$ and $\alpha \in \KK_q(k)$. Assume that $k$ satisfies $C_i(\alpha)$ and $[k:k^p] \leq p^{i+1}$. Then, every finitely generated regular extension $K/k$ of finite transcendence degree $\delta$ satisfies $C_{i+\delta}(\alpha|_K)$.
    \end{thm}
    \noindent Before proving the theorem, we need the following lemma.
    \begin{lemma}\label{lemma Weil restriction of hypersurfaces}
        Let $l/k$ be a finite extension and $H \subseteq \P^n_l$ a hypersurface of degree $d$. Then, there exist hypersurfaces $Z_1,\dots, Z_{m} \subseteq \P^{(n+1)m-1}_k$ of degree $d$ where $m=[l:k]$ such that for every extension $k'/k$ linearly disjoint from $l/k$ we have $H(l\otimes_k k') \neq \emptyset$ if and only if the intersection $Z \coloneq \bigcap_{i=1}^{m} Z_i$ admits a point over $k'$.
    \end{lemma}

    \begin{proof}
        Let $F(X_0,\hdots,X_n)$ be the polynomial defining $H$. Fix a basis $\omega_0,\dots, \omega_{m-1}$ of $l$ over $k$. We can replace the variable $X_i$ by $\underline{T}_i =T_{i,1}\omega_{i,0} + \dots + T_{i,m}\omega_{m-1}$. After rearranging the equation and expanding the constants, we see that
        \[
            F(\underline{T}_0,\hdots, \underline{T}_n)= \sum_{j=1}^{m}F_{j}(T_{0,1}, \hdots , T_{n,m}) \omega_{j},
        \]
        for certain homogeneous polynomials $F_j$ of degree $d$ with coefficients in $k$. For every $1\leq j \leq m$ we let $Z_j$ be the hypersurface in $\P^{(n+1)m-1}_k$ defined by $F_j =0$. It is clear that the intersection $Z$ has a point over $k'$ if $H(l\otimes_k k')\neq \emptyset$. Conversely, let us consider $[t_{0,1}: \dots : t_{n,m}] \in Z(k')$. Then, one can directly verify that 
        \[
        \left[\sum_{i=1}^{m}t_{0,i}\otimes \omega_0:\dots : \sum_{i=1}^{m}t_{n,i}\otimes \omega_i\right]\in \P^n(k'\otimes_k l)
        \] 
        lies in $H$. Note that we cannot have $\sum_{i=1}^{m}t_{0,i}\omega_j = 0$ for every $i \in \{0,\dots , n\}$ since the extensions $l/k$ and $k'/k$ are linearly disjoint.
    \end{proof}

    \begin{proof}[Proof of Theorem~\ref{thm constant transfer}]
        Let us first consider the purely transcendental case, that is $K=k(t_1,\dots, t_{\delta})$. In this case, the proof is easily reduced to the case $\delta = 1$. Let $Z \subseteq \mathbf{P}^n_{k(t)}$ be a hypersurface of degree $d$ such that $d^{i+1} \leq n$. Let $F \in k(t)[X_0,\dots , X_n]$ be a homogeneous polynomial defining $Z$. We may assume that $F$ has coefficients in $k[t]$. Let $M$ be the maximum degree in $t$ of the coefficients of $F$. Fix $s \in \mathbf{N}$. By explicit computation, we can find homogeneous polynomials $F_{0}, \hdots, F_{ds +M} \in k[T_{0,0}, \hdots, T_{n,s}]$ of degree $d$ such~that
        \[ F\left( \sum_{r=0}^s T_{0,r}t^r, \hdots , \sum_{r=0}^s T_{n,r}t^r\right) = \sum_{m=0}^{ds+M} F_m(T_{0,0}, \hdots, T_{n,s})t^m.\]
        For every $m \in \{0,\hdots, ds +M\}$ we denote by $Z_m \subseteq \mathbf{P}^{(n+1)(s+1)-1}_k$ the hypersurface defined by $F_m = 0$ and by $X$ the intersection $\bigcap_{m=0}^{sd+M} Z_m$. If we choose $s$ to be at least $(M+1)d^i - n$, we have $(sd+M+1)d^{i} \leq (s+1)(n+1)-1$ since $d^{i+1} \leq n$. By Lemma~\ref{lemma p-degree Ci}, the variety $X$ admits a point in a simple normal extension. Then, Proposition \ref{prop intersection hypersurfaces} ensures that $\alpha$ belongs to $\mathrm{N}_{q}(X/k)$. By \cite[Lemma 7.3.6]{GilleSzamuely2017CentralGalois} we conclude that $\alpha |_{k(t)}\in\mathrm{N}_q(Z/k(t))$. \par
        Note that the previous argument proves a slightly stronger property than the $C_{i+\delta}(\alpha|_K)$ property for $K$. It also proves that for every hypersurface $Z \subseteq \P^n_K$ of degree $d$ such that $d^{i+\delta}\leq n$, there exist finite extensions $k_1,\dots, k_r/k$ such that for every $j\in \{1,\dots , r\}$ we have $Z(k_jK) \neq \emptyset$ and $\alpha$ belongs to $\langle \Norm_q(k_j/k)\mid j\in \{1,\dots , r\}\rangle$. Moreover, we can follow the argument of Proposition~\ref{prop intersection hypersurfaces} to ensure that this also holds for intersections of hypersurface: for all hypersurfaces $Z_1,\dots, Z_m \subseteq \P^n_k$ of degree $d$ such that $md^{i+\delta} \leq n$ we can find finite extensions $k_1,\dots , k_r/k$ such that $Z(k_iK) \neq \emptyset$ and $\alpha \in \langle \Norm_q(k_j/k)\mid j\in \{1,\dots , r\}\rangle$ where $Z$ is $\bigcap_{j=1}^{m} Z_j$. \par

        Let $K/k$ be as in the statement and $t_1,\dots, t_{\delta} \in K$ be a separable transcendence basis of $K/k$. Fix $K_0$ to be $k(t_1,\dots , t_{\delta})$. Let $H\subseteq \P^n_{K}$ be a hypersurface of degree $d$ such that $d^{i+\delta} \leq n$ and $Z, Z_1,\dots , Z_{m} \subseteq \P_{K_0}^{(n+1)m -1}$ be as in Lemma~\ref{lemma Weil restriction of hypersurfaces}. Note that $[K_1:K_1^p] \leq p^{i+\delta +1}$ and $m d^{i+\delta} \leq (n+1)m-1$, so $Z$ admits a point in a normal simple extension of $K_0$ thanks to Lemma~\ref{lemma p-degree Ci}. By the previous paragraphs, there exist finite extensions $k_1,\dots, k_r/k$ such that $Z(k_jK_0) \neq \emptyset$ and $\alpha$ belongs to $\langle \Norm_q(k_j/k)\mid j\in \{1,\dots , r\}\rangle$. Now, since the extension $K/k$ is regular, $k_jK_0/K_0$ and $K/K_0$ are linearly disjoint for every $j\in \{1,\dots ,r\}$, see \cite[VIII~\S 4]{Lang2002Algebra} and \cite[Proposition~VIII~3.1]{Lang2002Algebra}. Thus, Lemma~\ref{lemma Weil restriction of hypersurfaces} ensures that $H$ admits a $k_j\otimes_{K_0} K$-point. This is enough to conclude that $\alpha|_{K} \in \Norm_q(H/K)$.
    \end{proof}

    \begin{rmk}\label{rmk only constant extensions}
        Note the argument above proves something slightly stronger than the $C_{i+\delta}(\alpha|_K)$ property for $K$. It proves that for all hypersurfaces $Z_1,\dots,Z_m \subseteq \P^n_K$ of degree $d$ such that $md^{i+\delta} \leq n$, there exist finite extensions $k_1, \dots , k_r/k$ such that the intersection $Z \coloneq \bigcap_{j=1}^m Z_j$ satisfies $Z(k_jK) \neq \emptyset$ and $\alpha \in \langle \mathrm{N}_q(k_j/k) \, | \, j \in \{1,\dots , r\} \rangle$.
    \end{rmk}

    \begin{rmk}\label{rmk hypersup with only non-simple points}
        Note that Theorem \ref{thm constant transfer} applies as soon as one can ensure that the intersection of hypersurfaces as in Proposition~\ref{prop intersection hypersurfaces} admits a point on a simple normal extension. We do not know if this is always the case. \par 
        The following example shows that the condition $r d^i \leq n$ is necessary. Let $K$ be a field such that $[K:K^p] \geq p^2$ and $a,b \in K^{\times}$ elements satisfying $[k^p(a,b):k^p] = p^2$. Let $Z_a$ and $Z_b$ be the hypersurfaces defined by 
        \begin{align*}
            Z_a: (X_0^p-aX_1^p)^p - a(X_2^p - a X_3^p)^p &= 0 \\
            Z_b: (X_0^p-bX_1^p)^p - b(X_2^p - b X_3^p)^p &= 0.
        \end{align*}
        Note that $Z_a$ and $Z_b$ are constructed as the first iteration of Lang's construction explained in Section~\ref{subsec: Lang's construction} applied to $\Phi = F_1= X_0^p-aX_1^p$ and $\Phi = F_1 = X_0^p-bX_1^p$, respectively. Then, thanks to Proposition~\ref{prop higher Lang's construction} the hypersurfaces $Z_a$ and $Z_b$ can only admit rational points in extensions containing $K(a^{1/p})$ and $K(b^{1/p})$ respectively. In particular, the intersection $Z \coloneq  Z_a \cap Z_b$ can only admit rational points in extensions containing $K(a^{1/p},b^{1/p})$. However, the intersection $Z$ is non-empty because $[a^{1/p}b^{1/p}:b^{1/p}:a^{1/p} : 1]$ is a solution over $K(a^{1/p}, b^{1/p})$.
    \end{rmk}
    
    \noindent When the field $k$ satisfies $[k:k^p] =p$, we can apply Theorem~\ref{thm constant transfer} for any extension $K/k$ such that $K$ does not contain any algebraic element over $k$ thanks to the following~lemma:
    \begin{lemma}\label{lemma regular p-degree 1}
        Let $k$ be a field of positive characteristic $p$ such that $[k:k^p] \leq p$ and $K/k$ a finitely generated field extension. Then, the extension $K/k$ is regular whenever $k$ is algebraically closed inside $K$.
    \end{lemma}
    \begin{proof}
        Since every extension over a perfect field is separably generated \cite[Corollary~VIII~4.4]{Lang2002Algebra}, we might assume that $[k:k^p] = p$. \par
        Let $t_1, \dots, t_{\delta} \in K$ be a transcendence basis of $K/k$ and $K_0 \coloneq k(t_1,\dots, t_{\delta})$. If the extension $K/K_0$ is separable, then the extension $K/k$ is regular by definition. Let $K_1/K_0$ be the maximal separable subextension of $K/K_0$. \par
        We first consider the case $K= K_1(a^{1/p})$ for some $a\in K_1 \setminus K^p$. Let us assume, for the sake of contradiction, that there exists a finite extension $l/k$ such that $l$ and $K$ are not linearly disjoint over $k$. Since the $l$ and $K_1$ are linearly disjoint over $k$, \cite[Proposition~VIII~\S~3]{Lang2002Algebra} ensures that $lK_1$ and $K$ are not linearly disjoint over $K_1$. In particular, $a^{1/p}$ belongs to $lK_1$ or equivalently, $K \subseteq lK_1$. Additionally, there exists a purely inseparable subextension $k'/k$ of $l/k$ such that $l/k'$ is separable. Indeed, the extension $l/k$ can be decomposed as $l_s/k$ for a separable extension and $l=l_s(x^{1/p^r})$ for a well-chosen $x\in k$. Since $[k:k^p] = p$, we can take $k' = k(x^{1/p^r})$.
        \begin{equation*}
            \begin{tikzcd}[column sep =3ex , row sep = 4ex]
                & lK_1 \ar[dl,dash] \ar[dr,dash,"k'K_1" description]& & K \ar[ll,hookrightarrow]\ar[dl,dash]\\
                l\ar[rd,dash, "k'" description]& & K_1\ar[dl,dash] &\\
                & k & &
            \end{tikzcd}
        \end{equation*}
        Since the extension $lK_1/k'K_1$ is separable, we must have $K \subseteq k'K_1 = K_1(x^{1/p^r})$. Then, there must exist $\lambda_0, \dots , \lambda_{p^r -1} \in K_1$ such that 
        \[
            a^{1/p} = \sum_{i=0}^{p^r-1} \lambda_i x^{i/p^r}.
        \]
        We can take the $p$-th power of the previous expresion to obtain the following:
        \begin{align*}
            a & = \sum_{i=0}^{p^r-1} \lambda_i^p x^{i/p^{r-1}} \\
            &= \sum_{s=0}^{p^{r-1}-1}\left(\sum_{j=0}^{p-1}\lambda_{p^{r-1}j+s}^px^j\right) x^{s/p^{r-1}}.
        \end{align*}
        Since $a$ belongs to $K_1$, we must have $\sum_{j=0}^{p-1}\lambda_{p^{r-1}j+s}^px^j = 0$ for every $s \neq 0$. Moreover, the elements $1,x,\dots , x^{p-1}$ are linearly independent over $K_1^p$ because the extension $K_1/k$ is separable. We deduce that $\lambda_{p^{r-1}j+s} = 0$ for every $s \in \{1,\dots , p^r-1\}$ and in particular,
        \[
            a^{1/p} = \sum_{j=0}^{p-1}\lambda_{p^{r-1}j} x^{j/p}.
        \]
        This leads to a contradiction because it implies that $K=K_1(x^{1/p})$, but $x^{1/p}$ algebraic over $k$. We deduce that $K/k$ is linearly disjoint from $\overline{k}/k$ and thus, it is regular. \par
        Let $K/K_1$ be a general purely inseparable extension. There exists a tower of subextensions $K=K_n \supseteq K_{n-1}  \supseteq \dots \supseteq K_1$ such that $K_i = K_{i-1}(a_i^{1/p})$ for $a_i \in K_i \setminus K_i^p$. The previous argument ensures that $K_2/k$ is a regular extension. An inductive argument allows us to conclude that $K/k$ is also a regular extension.
    \end{proof}
    
    \noindent In the case $q=0$, we can deduce the optimal transfer result from \cite[Theorem~2a]{Nagata1957NotesPaperLang}.
    \begin{prop}\label{prop transfer C10}
        Let $K/k$ be a field extension of finite transcendence degree $j$. Assume that $k$ satisfies $C_i^0$. Then $K$ satisfies $C_{i+j}^0$. 
    \end{prop}
    \begin{proof}
        Let $L/K$ be a finite extension. Denote by $l$ the algebraic closure of $k$ in $L$. Let $Z \subseteq \P^n_L$ be a hypersurface of degree $d$ such that $d^{i+j} \leq n$. By \cite[Lemma~1]{KatoKuzumaki1986Dimension} it suffices to prove that for every prime number $\ell$ the hypersurface $Z$ admits a rational point on a finite extension $L_{\ell}/L$ of degree prime to $\ell$. Let $l(\ell)/l$ be a maximal prime to $\ell$ extension. By \cite[Lemma~1]{KatoKuzumaki1986Dimension} we have that $l(\ell)$ is $C_i$. We can use \cite[Theorem 2a]{Nagata1957NotesPaperLang} to deduce that the composite extension $l(\ell)L$ is $C_{i+j}$. In particular, there exists a finite subextension $L_{\ell}/L$ of $l(\ell)L / L$ such that $Z(L_{\ell})\neq\emptyset$. Since any subextension of $l(\ell)L / L$ has degree prime to $\ell$ and this can be done for every prime number, we conclude that $Z$ admits a zero-cycle of degree $1$. 
    \end{proof}
        
    \noindent Unlike the $q=0$ case, we cannot expect that ``if $k$ satisfies $C_i^q$, then $k(x)$ and $k(\!(x)\!)$ satisfies $C_{i+1}^q$'' for $q \geq 1$. To see this we need the following lemma:
    \begin{lemma}\label{lemma going-down KK properties}
        Let $k$ be a field and $p$ be either a prime number or $1$. Assume that the field of rational functions $k(x)$ or the field of Laurent series $k(\!(x)\!)$ satisfies $C_{i+1}^q$ away from $p$ with $q \geq 1$. Then, $k$ satisfies $C_{i+1}^{q-1}$ away from $p$.
    \end{lemma}
    \begin{proof}
    Since the proof for $k(\!(t)\!)$ is very similar, we only prove the statement for $k(x)$. \par
    Let $l/k$ be a finite extension and $Z \subseteq \P^n_l$ a hypersurface of degree $d$ such that $d^{i+1}\leq n$ and $\alpha \in \KK_{q-1}(l)$. Since the residue map $\partial:\KK_q(l(x)) \to \KK_{q-1}(l)$ with respect to the $x$-adic valuation is surjective, there exists a class $\beta \in \KK_q(l(x))$ such that $\partial(\beta)=\alpha$. Moreover, there exist finite extensions $L_1,\hdots,L_m/l(x)$ and $\gamma_j \in \KK_q(L_j)$ for every $j \in \{1,\dots , m\}$ such that $Z(L_j)\neq \emptyset$ and $p^s\beta = \prod_{j=1}^m\mathrm{N}_{L_j/l(x)}(\gamma_j)$ for some $s\in \NN$ because $l(x)$ satisfies $C_{i+1}^q$ away from $p$. The $x$-adic valuation extends to discrete valuations $v_{j,1},\hdots , v_{j,r_j}$ over $L_j$. Denote by $l_{j,u}$ the residue field of $L_{j}$ and by $\partial_{j,u}:\KK_q(L_j) \to \KK_{q-1}(l_{j,u})$ the residue map both with respect to the valuation $v_{j,u}$. By compatibility of the residues with the norm \cite[Corollary 7.4.3]{GilleSzamuely2017CentralGalois}, we have that 
    \begin{equation*}
        p^s \alpha = \prod_{j=1}^m \prod_{u=1}^{r_j}\mathrm{N}_{l_{j,u}/l}(\partial_{j,u}(\gamma_j)).
    \end{equation*}
    Moreover, the valuative criterion for properness ensures that $Z(l_{j,u})\neq \emptyset$, and hence the result follows.
    \end{proof}
        \begin{rmk}
        The optimistic statement ``if $k$ satisfies $C_i^q$, then $k(x)$ and $k(\!(x)\!)$ satisfies $C_{i+1}^q$'' is false for $q \geq 1$. Indeed, assume that this statement is true. In \cite{CTMadore2004NotC10} we can find a construction of a field of cohomological dimension $1$ that is not $C_1^0$, but the optimistic statement together with Lemma \ref{lemma going-down KK properties} ensure that $C^1_0$ implies $C_1^0$. \par
        A more reasonable expectation would be that if $k$ satisfies $C_i^q$ and $C_{i-1}^{q+1}$, then $k(x)$ and $k(\!(x)\!)$ satisfies $C_{i+1}^q$. In the latter case this is known when $k$ is perfect and when restricting to hypersurfaces of prime degree by \cite[Theorem 1 (3)]{KatoKuzumaki1986Dimension}. Or less optimistically, that if $k$ satisfies Kato-Kuzumaki's conjecture for all non-negative integers $i,q$, then so do $k(x)$ and $k(\!(x)\!)$.
    \end{rmk}
    \subsection{Transfer results for non-constant classes}
    Complementary to the previous section, in this section we consider the behaviour of the non-constant classes of transcendental extensions. Namely, we have the following theorem:
    \begin{thm}\label{thm non-contant transfer}
        Suppose that $k$ satisfies $C_i^q$. If $k$ has positive characteristic we further assume that $[k:k^p] \leq p^{i+1}$. Then, the field $k(x)$ satisfies $C_{i+1}(\alpha)$ for every element $\alpha \in \mathrm{K}_{q+1}(k(x))$.
    \end{thm}
    \begin{proof}
        Let $Z \subseteq \mathbf{P}^n_{k(x)}$ be a hypersurface of degree $d$ with $d^{i+1}\leq n$, and $\alpha \in \mathrm{K}_{q+1}(k(x))$. We use the notation described in Section \ref{sec Milnor rational functions}. Let $p_1,\hdots , p_n \in \mathbf{A}^1_{k,(0)}$ and $\beta _j \in \mathrm{K}_q(k(p_j))$ be such that $\partial_{p_j}(\alpha) =  \beta_j$ and $\partial_p(\alpha) = 1$ for any $p\not \in \{p_1,\hdots, p_n\}$. We know that $k(p_j)(x)$ satisfies $C_{i+1}(\beta_j |_{k(p_j)(x)})$ for every $1\leq j \leq n$, moreover, we only need constant extensions thanks to Remark \ref{rmk only constant extensions}. Then, for every $1\leq j\leq n$ there exist finite extensions $l_{j,1}, \hdots , l_{j,m_j} / k(p_j)$ and elements $\gamma_{jr} \in \mathrm{K}_{q}(l_{jr})$ such that $Z(l_{jr}(x)) \neq \emptyset$ and 
        \[ 
        \beta_j = \prod_{r=1}^{m_j} \mathrm{N}_{l_{jr}/k(p_j)}(\gamma_{jr}).
        \]
        Milnor's exact sequence ensures that the map 
        \[
        \partial: \mathrm{K}_{q+1}(l_{jr}(x)) \to \bigoplus_{q \in \mathbf{A}^1_{l_{jr},(0)}}\mathrm{K}_{q}(l_{jr}(q))
        \]
        is surjective. In particular, for every $1\leq j \leq n$ and $1\leq r\leq m_j$ there exists a class $\tilde{\gamma}_{jr} \in \mathrm{K}_{q+1}(l_{jr}(x))$ such that $\partial_{p_j'}(\tilde{\gamma}_{jr}) = \gamma_{jr}$ and $\partial_p(\tilde{\gamma}_{jr}) = 1$ for any $p\neq p_j'$ where $p_j'$ is a rational point of $\mathbf{A}^1_{l_{jr}}$ above $p_j$. We deduce that 
        \begin{align*}
        \partial_{p_j}\left( \prod_{r=1}^{m_j} \mathrm{N}_{l_{jr}(x)/k(x)}(\tilde{\gamma}_{jr})\right) & = \prod_{r=1}^{m_j} \mathrm{N}_{l_{jr}/k(p_j)}(\partial_{p_j'}(\tilde{\gamma}_{jr}))\\
        &= \prod_{r=1}^{m_j} \mathrm{N}_{l_{jr}/k(p_j)}(\gamma_{jr}) \\
        &= \beta_j.
    \end{align*}
    Moreover, for any $p\neq p_j$ we have $\partial_{p}\left( \prod_{r=1}^{m_j} \mathrm{N}_{l_{jr}(x)/k(x)}(\tilde{\gamma}_{jr})\right) = 1$. Therefore by Milnor's exact sequence we have that 
    \[\alpha \cdot \left(\prod_{j=1}^n \prod_{r=1}^{m_j} \mathrm{N}_{l_{jr}(x) /k(x)}(\tilde{\gamma}_{jr}) \right)^{-1}\in \mathrm{K}_q(k).\]
    Since $k$ also satisfies $C_i^{q+1}$, by Theorem~\ref{thm constant transfer} we have that $k(x)$ satisfies $C_{i+1}(\alpha_0|_{k(x)})$ for any $\alpha_0 \in \mathrm{K}_{q+1}(k)$, and hence we conclude that $\alpha$ belongs to $\mathrm{N}_{q+1}(Z/k(x))$.
    \end{proof}
    \begin{rmk}
     The previous proof also works for a more general family of curves, but the hypothesis is difficult to verify. Let $k$ be a field and $C$ a geometrically integral smooth curve over $k$. The argument of Theorem \ref{thm non-contant transfer} also works when the map
    \[ \KK_{q+1}(k(C)) \to \bigoplus_{p\in C_{(0)}} \KK_{q}(k(p))\]
    is surjective and its kernel is the image of $\KK_q(k)$.
    \end{rmk}

\subsection{Transfer Principles for Laurent Series}

In this section we obtain an analogue of Theorem \ref{thm constant transfer} for Laurent series. To prove this result, we need the following refinement of \cite[Theorem~1]{Greenberg1966RationalPointsHenselian}.
    \begin{lemma}\label{lemma refined Greenberg}
        Let $R$ be a Henselian discreetly valued field, $K$ its fraction field, and $t$ a uniformiser. Let $\underline{F}=(F_1,\dots,F_r)$ be an $r$-tuple of polynomials in $R[X_1,\dots, X_n]$. Denote by $I$ the ideal generated by $\underline{F}$ in $R[X_1,\dots,X_n]$. If $K$ has positive characteristic, we further assume that the extension $\hat{K}/K$ is separable where $\hat{K}$ is the completion of $K$. Then, there are integers $N\geq 1$, $c \geq 1$ and $s\geq 0$ only depending on $I$ such that for every unramified extension $R'/R$, integer $\nu\geq N$ and $x_1,\dots, x_n \in R'$ such that 
        \[
            \underline{F}(x_1,\dots,x_n) \equiv 0  \mod t^{\nu},
        \]
        there exists $y_1,\dots, y_n \in R'$ such that\begin{align*}
            y_i & \equiv x_i \mod t^{\lfloor\nu/c\rfloor - s} \\
            0 & = \underline{F}(y_1,\dots, y_n).
        \end{align*}
    \end{lemma}
    \noindent Lemma~\ref{lemma refined Greenberg} can be proved in an identical manner to that of \cite[Theorem~1]{Greenberg1966RationalPointsHenselian} by strengthening the inductive hypothesis.
\begin{cor}\label{cor constant transfer Laurent}
        Let $k$ be a field. Let $q$ be a non-negative integer and $\alpha \in \mathrm{K}_q(k)$. Assume that $k$ satisfies $C_i(\alpha)$ for some non-negative integer $i$. Then, $K\coloneq k(\!(t)\!)$ satisfies $C_{i+1}(\alpha|_{K})$. If we assume $[k:k^p]\leq p^{i}$ when $k$ is of positive characteristic $p$, then $K$ satisfies $C_{i+1}(\beta)$ for every $\beta \in \mathrm{U}_q(K)$ that maps to $\alpha$ along the specialisation map $s:\mathrm{U}_q(K) \to \mathrm{K}_q(k)$.
    \end{cor}
    \begin{proof}
        Let $Z \subseteq\mathbf{P}^n_K$ be a hypersurface of degree $d$ such that $d^{i+1} \leq n$. For every $\nu\in \mathbf{N}$ we can find a hypersurface $Z^{(\nu)}\subseteq \mathbf{P}^n_{k(t)}$ of degree $d$ with the same reduction modulo $t^{\nu}$ as $Z$. By Theorem~\ref{thm constant transfer} and Remark~\ref{rmk only constant extensions} there exist finite extensions $l_1,\hdots , l_r /k$ and elements $\alpha_j \in \mathrm{K}_q(l_j)$ for every $1\leq j\leq r$ such that $Z^{(\nu)}(l_j(t)) \neq \emptyset$ and $\alpha = \prod_{j=1}^r \mathrm{N}_{l_j/k}(\alpha_j)$. By Lemma~\ref{lemma refined Greenberg}, for sufficiently large $\nu$ we have that $Z(l_j(\!(t)\!)) \neq \emptyset$, for every $1\leq j \leq r$. By \cite[Lemma 7.3.6]{GilleSzamuely2017CentralGalois} we deduce that $\alpha|_{K}\in\mathrm{N}_q(Z/K)$. \par 
        Keeping the notation from the previous paragraph and assuming that $k$ satisfies $[k:k^p]\leq p^{i}$ if it has characteristic $p$, we claim that
        \begin{equation*}
            \mathrm{U}_{q}^1(K)\subset \operatorname{N}_q(Z/K).
        \end{equation*}
         Indeed, if $k$ is characteristic zero, then since $\mathrm{U}_{q}(K)$ is $\ell$-divisible for all primes $\ell$, we have $\mathrm{U}_{q}^1(K)\subset \operatorname{N}_q(Z/K)$. Moreover, if $k$ has positive characteristic $p$, then by Proposition~\ref{prop separable closed Cdelta} and \cite[Theorem~2]{Greenberg1966RationalPointsHenselian} the field $k^s\lau{t}$ is $C_{i+1}$, in particular, we have that $Z^{(\nu)}(k^{s}\lau{t})\neq\varnothing$ for all $\nu\in \NN$, and hence we have $Z(k^{s}\lau{t})\neq\varnothing$. Therefore by Lemma~\ref{lem: Z has K^sep point implies U^1(K) norm of Z} we have $\mathrm{U}_{q}^1(K)\subset \operatorname{N}_q(Z/K)$. \par
        
    Now let $\beta \in \mathrm{U}_q(K)$ be such that $s(\beta) = \alpha$. We have the following commutative diagram with exact rows
        \begin{equation} \label{diag norms and unity groups}
            \begin{tikzcd}
                0 \ar[r] & \prod_{j=1}^r\mathrm{U}_{q}^1(l_j(\!(t)\!)) \ar[r] \ar[d] & \prod_{j=1}^r\mathrm{U}_{q}(l_j(\!(t)\!)) \ar[r] \ar[d, "\prod_{j=1}^r\mathrm{N}_{l_j(\!(t)\!)/K}"] & \prod_{j=1}^r\mathrm{K}_{q}(l_j) \ar[r] \ar[d, "\prod_{j=1}^r  \mathrm{N}_{l_j/k}"] & 0 \\
                0 \ar[r] & \mathrm{U}_{q}^1(K) \ar[r]  & \mathrm{U}_{q}(K) \ar[r]  & \mathrm{K}_{q}(k) \ar[r] & 0.
            \end{tikzcd}
        \end{equation}
  In particular, for every $1\leq j \leq r$ there exists a class $\beta_j \in \KK_q(l_j\lau{t})$ such that $s(\beta_j) = \alpha_j$ for $\alpha_j$ as above. A simple diagram chase in diagram \ref{diag norms and unity groups} ensures that $\beta^{-1} \cdot \prod_{j=1}^r \mathrm{N}_{l_j\lau{t}/K}(\beta_j)$ belongs to $\mathrm{U}^1_q(K)$. Therefore since $\mathrm{U}^1_q(K)\subset \mathrm{N}_q(Z/K)$ the result follows.
    \end{proof}

    \begin{rmk}\label{rmk only constant extensions Laurent}
        Note that in the same spirit of Remark \ref{rmk only constant extensions} we only need constant extensions. Keep the notation from Corollary~\ref{cor constant transfer Laurent}. The last proof ensures that $K$ satisfies the following property: For any hypersurface $Z \subseteq \mathbf{P}^n_k$ of degree $d$ such that $d^{i+1}\leq n$ there exists a family $l_1,\hdots , l_r/k$ of finite extensions and classes $\beta_j \in \mathrm{K}_q(l_j)$ such that $Z(l_j(\!(t)\!))\neq \emptyset$ for every $1\leq j\leq r$ and $\alpha = \prod_{j=1}^{r} \mathrm{N}_{l_j/k}(\beta_j)$.
    \end{rmk}

    \noindent A slightly simpler proof than Theorem~\ref{thm non-contant transfer} gives a stronger result for Laurent series in characteristic zero.
    \begin{prop}\label{prop non-constant transfer Laurent}
        Let $k$ be a field of characteristic zero. Assume that $k$ satisfies $C_i^q$. Then, the field $k(\!(x)\!)$ satisfies $C_{i+1}^{q+1}$.
    \end{prop}
    \begin{proof}
        Any finite extension of $k(\!(x)\!)$ is of the form $l(\!(t^{1/m})\!)$ for a finite extension $l/k$, a uniformiser $t$ of the unique extension of the $x$-adic valuation, and a positive integer $m$. We can simplify the notation and set $l=k$, $x=t$ and $m=1$. Consider $\alpha \in \KK_{q+1}(k(\!(x)\!))$. Let $Z \subseteq \mathbf{P}^n_{k(\!(x)\!)}$ be a hypersurface of degree $d$ such that $d^{i+1} \leq n$. By Remark~\ref{rmk only constant extensions Laurent} there exist finite extensions $l_1,\hdots, l_m/k$ and $\beta_j \in \KK_q(l_j)$ such that $Z(l_j(\!(x)\!))\neq \emptyset$ and $\partial(\alpha) = \prod_{j=1}^m\mathrm{N}_{l_j/k}(\beta_j)$. Since the residue map $\partial: \KK_{q+1}(l_j(\!(x)\!)) \to \KK_q(l_j)$ is surjective we can find elements $\tilde{\beta_j} \in \KK_q(l_j(\!(x)\!))$ such that $\partial(\tilde{\beta_j})= \beta_j$. In particular we have
        \[  \alpha \cdot \prod_{j=1}^m \mathrm{N}_{l_j(\!(x)\!)/k(\!(x)\!)}(\tilde{\beta_j}^{-1}) \in \mathrm{U}_{q+1}(k(\!(x)\!)),\]
        and hence by Corollary~\ref{cor constant transfer Laurent} the result follows.
    \end{proof}

    \noindent The proof of Proposition \ref{prop transfer C10} can be adapted to obtain the following result without characteristic restrictions:
    \begin{prop}
        Assume that $k$ satisfies $C_i^0$. Then the field of Laurent series $k(\!(t)\!)$ satisfies~$C_{i+1}^0$.
    \end{prop}

\section{Applications}
\label{sec: Applications}
    In this section we apply the results from the previous sections to obtain new cases for which the $C^q_i$ property holds. By using Theorem~\ref{thm constant transfer} we obtain the following general transfer result:
    \begin{thm}\label{thm transfer applications}
        Let $i,q$ be a pair of non-negative integers and $k$ a field of exponent characteristic $p$ such that $k(x_1,\dots, x_q)^{\mathrm{perf}}$ satisfies $C_i^q$. Then, for all non-negative integers $j$ and $q' \leq q$, the field $k(x_1,\dots, x_{q'+j})$ satisfies the $C_{i+j}^{q'}$ property away from $p$.
    \end{thm}
    \begin{proof}
        It would be enough to prove that for every non-negative integer $j$, the field $k(x_1,\dots, x_{q+j})$ satisfies $C_{i+j}^q$ away from $p$. Indeed, a direct application of Lemma~\ref{lemma going-down KK properties} yields that the $C_{i+j}^q$ property away from $p$ for $k(x_1,\dots, x_{q+j})$ implies the $C_{i+j}^{q'}$ property away from $p$ for the field $k(x_1,\dots, x_{q'+j})$.\par
        We consider all algebraic extensions as subextensions of a fixed algebraic closure of $k(x_1,\hdots,x_{q+j})$. Let $K/k(x_1,\hdots,x_{q+j})$ be a finite extension and $\alpha \coloneq \{f_1,\hdots , f_q \} \in \KK_q(K)$ be a non-trivial symbol. Assume that $f_1,\hdots ,f_r$ are algebraically independent over $k$ and $f_{r+1},\hdots ,f_q$ are algebraic over $k(f_1,\hdots , f_r)$. So that $k(f_1,\hdots,f_q)$ is an extension of transcendence degree $r$ over $k$. We can find $g_{r+1},\hdots ,g_q \in K$ such that the extension $K(f_1,\dots, f_q, g_{r+1},\hdots, g_q)/k$ has transcendence degree $q$. Let $K_0$ be the algebraic closure of $K(f_1,\dots, f_q, g_{r+1},\hdots, g_q)$ in $K$. \par
        Fix $\alpha_0 \coloneq \{f_1,\hdots , f_q \} \in \KK_q(K_0)$ and note that $\alpha_0 |_K = \alpha$. By our hypothesis, the field $K_0^{\mathrm{perf}}$ satisfies $C_i(\alpha_0|_{K_0^{\mathrm{perf}}})$. Since the extension $KK_0^{\mathrm{perf}}/K_0^{\mathrm{perf}}$ is finitely generated, regular, and has transcendence degree $j$, Theorem~\ref{thm constant transfer} ensures that $KK_0^{\mathrm{perf}}$ satisfies $C_{i+j}(\alpha_0|_{KK_0^{\mathrm{perf}}})$. \par

        The fact that $KK_0^{\mathrm{perf}}$ satisfies $C_{i+j}(\alpha_0|_{KK_0^{\mathrm{perf}}})$ implies the following: for every hypersurface $Z \subseteq \P^n_{K}$ of degree $d$ such that $d^{i+j} \leq n$, there exist finite extensions $K_1,\dots ,K_r/KK_0^{\mathrm{perf}}$ and $\beta_j \in \KK_q(K_j)$ such that $Z(K_j)\neq \emptyset$ and 
        \[
        \alpha_0|_{K K_0^{\mathrm{perf}}}=\prod_{j=1}^r \Norm_{K_j/KK_0^{\mathrm{perf}}}(\beta_j).
        \] 
        Thus, there exists a large enough finite purely inseparable extension $K_0'/K_0$ together with finite extensions $K_1', \dots, K_r'/KK_0'$ and classes $\beta_j' \in \KK_q(K_j')$ such that for every $j\in \{1,\dots , r\}$ we have
        \begin{itemize}
            \item  $\beta_j' |_{K_j} = \beta_j$, 
            \item  $Z(K_j') \neq \emptyset$, and
            \item  $K_j' \otimes_{KK_0'} KK_0^{\mathrm{perf}}$ is a field.
        \end{itemize}
        Note that the tensor product $K_j' \otimes_{KK_0'} KK_0^{\mathrm{perf}}$ gets naturally identified with $K_j$. Then, \cite[Lemma~7.3.6]{GilleSzamuely2017CentralGalois} ensures that the following diagram is commutative
        \begin{equation*}
            \begin{tikzcd}[row sep = 5ex , column sep = 13ex]
                \prod_{j=1}^r\mathrm{K}_{q}(K_j') \ar[r,"\prod_j \mathrm{N}_{K_j'/KK_0'}"] \ar[d, "\prod_j\iota_j"]&  \mathrm{K}_{q}(KK_0') \ar[d, "\iota"] \\
                \prod_{j=1}^r\mathrm{K}_{q}(K_j) \ar[r, "\prod_j \mathrm{N}_{K_j/KK_0^{\mathrm{perf}}}"]& \mathrm{K}_{q}(KK_0^{\mathrm{perf}}).
            \end{tikzcd}
        \end{equation*}
        We deduce that
        \[
        \alpha_0|_{KK_0'} = \prod_{j=1}^s\Norm_{K_j'/KK_0'}(\beta_j')
        \] 
        because the extension $KK_0^{\mathrm{perf}}/KK_0'$ is purely inseparable and thus, the restriction map $\iota: \mathrm{K}_{q}(KK_0') \to \mathrm{K}_{q}(KK_0^{\mathrm{perf}})$ is injective as stated in Proposition~\ref{prop norm purely inseparable}.\par
        Since $\Norm_{KK_0'/K}(\alpha_0|_{KK_0'}) = [K_0':K_0] \alpha$, we conclude that $[K_0':K_0]\alpha$ belongs to $\Norm_q(Z/K)$. This is enough to conclude that $k(x_1,\dots , x_{q+j})$ satisfies the $C_{i+j}^q$ property away from~$p$.
    \end{proof}
\noindent The following diagram gives a graphic representation of these relations:
    \begin{equation*}
    \begin{tikzcd}[row sep = 3ex, column sep = 3ex] 
    {} & & & & & \\
    C^{i+q}_0 \arrow[u] \arrow[rd, "\ddots", phantom]  & & & & &\\
    & \mathrm{perf}\; C_i^q  \arrow[r, Rightarrow] \arrow[rrr, "\cdots" description, Rightarrow, bend left] \arrow[d, Rightarrow] & C_{i+1}^q \arrow[d, Rightarrow] & \cdots & C_{i+j}^q \arrow[d, Rightarrow] &    \\
    & C_{i}^{q-1} \arrow[dd,"\vdots" description, Rightarrow]& C_{i+1}^{q-1} \arrow[dd,"\vdots" description, Rightarrow] & \cdots  &  C_{i+j}^{q-1} \arrow[dd,"\vdots" description, Rightarrow]& \\
    C_0^i \arrow[rd, "\ddots", phantom] \arrow[uuu, "\vdots" description, no head] &  &  & &  &    \\
    {} \arrow[r,start anchor=center, no head] \arrow[u,start anchor=center, no head] & C_i^0 \arrow[r, no head] \arrow[rd, no head, dotted] & \cdots \arrow[r, no head]& C_{i+q-1}^0 \arrow[r,no head] \arrow[rd, no head, dotted] & C_{i+q}^0 \arrow[r] \arrow[rd, no head, dotted]& {} \\
    & & k &  & k(x_1,\dots,x_{q-1})    & k(x_1,\dots,x_q)  
    \end{tikzcd}
    \end{equation*}

\noindent In the above diagram, the point $(a,b)$ represents the $C_a^b$ property away from $p$ for $k(x_1,\cdots, x_{a+b-i})$ for non-negative integers $(a,b) \neq (i,q)$ such that $a+b \geq i$ and the point $(i,q)$ represents the $C_i^q$ property for $k(x_1,\dots , x_q)^{\mathrm{perf}}$.

    \noindent Moreover, thanks to Lemma~\ref{lemma regular p-degree 1} we can obtain the following result for $q=1$.

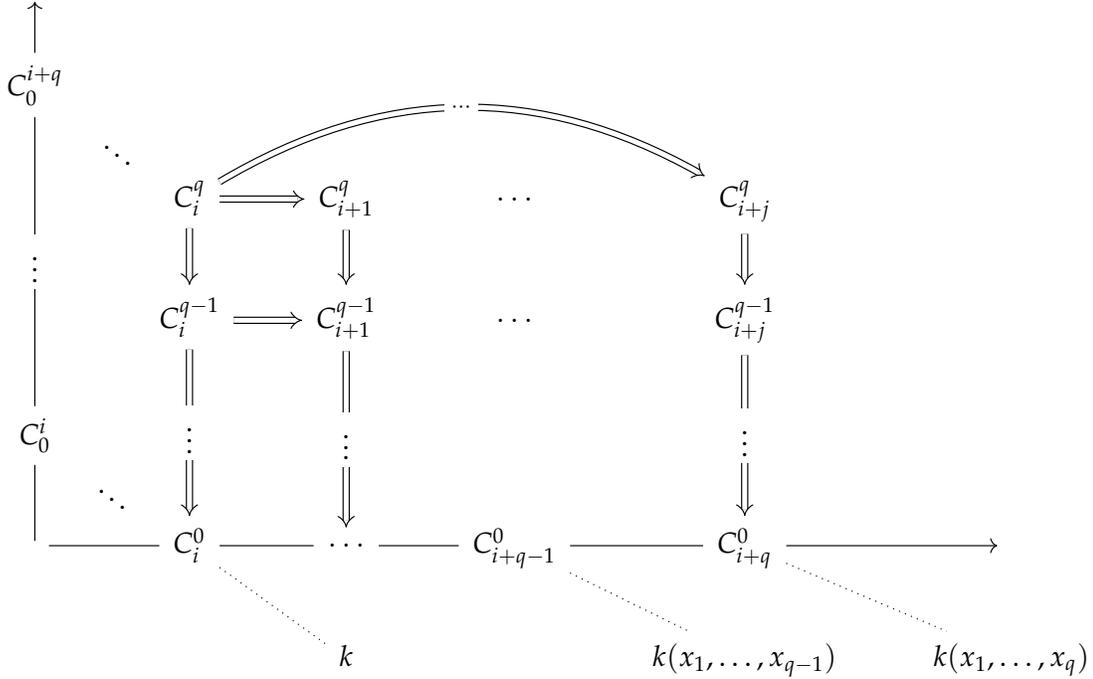
    \begin{thm}\label{thm transfer applications p-degree 1}
        Let $i,q$ be a pair of non-negative integers and $k$ a perfect field such that $k(x)$ satisfies $C_i^1$. Then, for every non-negative integer $j$, the field $k(x_1,\dots, x_{j+1})$ satisfies the $C_{i+j}^{1}$~property.
    \end{thm}
    \noindent The argument is similar to that of Theorem~\ref{thm transfer applications}, but we include it for the convenience of the~reader.
    \begin{proof}
        Let $K/k(x_1,\dots , x_{j+1})$ be a finite extension and $f\in K^{\times}$. Let $\tilde{k}$ be the algebraic closure of $k(f)$ in $K$ when $f$ is not algebraic over $k$ and the algebraic closure of $k(x_1)$ when $f$ is algebraic over $k$. In either case, the field $\tilde{k}$ satisfies the $C_i(f)$ property by hypothesis. Moreover, the extension $K/\tilde{k}$ is an extension of transcendence degree $j$ such that $\tilde{k}$ is algebraically closed in $K$. Then, Lemma~\ref{lemma regular p-degree 1} ensures that $K/\tilde{k}$ is a regular and Theorem~\ref{thm constant transfer} allow us to conclude that $K$ satisfies $C_{i+j}(f)$. This is enough to conclude.
    \end{proof}
    \begin{rmk}
        If we could replace the ``regularity'' hypothesis in Theorem~\ref{thm constant transfer} by ``$k$ is algebraically closed in $K$'', the argument above would yield the following result: let $k$ be a field of characteristic exponent $p$. Assume that $[k:k^p] \leq p^{i-q+1}$. If the field $k(x_1,\hdots, x_q)$ satisfies $C_i^q$, then for all non-negative integers $j$ and $q ' \leq q $, the field $k(x_1,\dots, x_{q'+j})$ satisfies $C_{i+j}^{q'}$.
    \end{rmk}

    \noindent The main corollary of Theorem~\ref{thm transfer applications p-degree 1} is for function fields over finite fields. By combining Proposition~\ref{prop global fields C11} with the theorem, we deduce the following result:
    \begin{cor}\label{cor Cn1 function finite}
        The field $\mathbf{F}_p(x_1,\dots , x_n)$ satisfies $C_n^1$.
    \end{cor}

    \noindent There are several cases where the hypothesis of Theorem~\ref{thm transfer applications} is not known to hold, but Theorem~\ref{thm constant transfer}, Theorem~\ref{thm non-contant transfer}, Corollary~\ref{cor constant transfer Laurent}, and Proposition~\ref{prop non-constant transfer Laurent} still give partial results towards the expected Kato-Kuzumaki properties. The remainder of this section is devoted to explain some of these applications.

    \subsection{Finite fields}
    \noindent We cannot apply Theorem~\ref{thm transfer applications} to deduce the $C_n^2$ property away from $p$ for the field $\mathbf{F}_p(x_1,\dots, x_{n+1})$ because the $C_1^2$ property for $\mathbf{F}_p(x,y)^{\mathrm{perf}}$ is still unknown. However, we also can apply Theorem~\ref{thm constant transfer} to deduce the following result because $\mathbf{F}_p(x)$ satisfies $C_0^2$.
    \begin{prop}\label{prop partial Cn-1 2 function over finite}
        Let $K/ \mathbf{F}_q(x_1,\dots, x_n)$ be a finite extension. Let $f$ be in $K^{\times}$ and denote by $\tilde{k}$ the algebraic closure of $\mathbf{F}_q(f)$ in $K$. Then, if the extension $K/\tilde{k}$ is regular, the field $K$ satisfies $C_{n-1}(\alpha)$ for every $\alpha \in \mathrm{im}\left(\KK_2(\tilde{k}) \to \KK_2(K)\right)$.
    \end{prop}

    \noindent It should be noted that the group $\KK_2(K)$ is not generated by the images of $\KK_2(\tilde{k})$ when $f$ ranges over $K^{\times}$. Even when $K= \mathbf{F}_p(x,y)$ this is not the case. Indeed, \cite[Theorem~7]{CadoretPirutka2021Reconstructing} ensures that the groups $\KK_2(\tilde{k})$ is a torsion group for $\tilde{k}$ as in Proposition~\ref{prop partial Cn-1 2 function over finite} while the class $\{x,y\} \in \KK_2(K)$ cannot be torsion again by \cite[Theorem~7]{CadoretPirutka2021Reconstructing}.

    \subsection{\texorpdfstring{$p$}{p}-adic fields}
        \noindent When the base field is a $p$-adic field we may combine the main results in \cite{DiegoLuco2024KKp-adicFunction} with our results to obtain the following:
        \begin{prop}\label{prop p-adic function field C2n+1}
            Let $k$ be a $p$-adic field and $X$ a geometrically integral variety over $k$ of dimension $n$. Denote by $K$ the function field of $X$. Fix an element $f\in K^{\times}$ and denote by $\tilde{k}$ the algebraic closure of $k(f)$ in $K$, then $K$ satisfies $C_{n+1}(\alpha)$ for every $\alpha \in \mathrm{im}\left( \KK_2(\tilde{k})\to \KK_2(K)\right)$.
        \end{prop}
        \begin{proof}
            The field $\tilde{k}$ satisfies $C_2^2$ thanks to \cite[Main Theorem 1]{DiegoLuco2024KKp-adicFunction}. We can conclude directly from Theorem~\ref{thm constant transfer}.
        \end{proof}
        \noindent This gives a partial result towards the $C_3^2$ property for $K:= \mathbf{Q}_p(x,y)$, but some elements of $\KK_2(K)$ are missing. Most notably, we do not know whether or not $K$ satisfies $C_3(\{x,y\})$. \par
        We can also use \cite[Main Theorem 2]{DiegoLuco2024KKp-adicFunction} to get a statement in the same spirit as Proposition \ref{prop p-adic function field C2n+1}. Before giving the exact result, we state a weaker result to illustrate its use.
        \begin{prop}
            Let $k$ be a $p$-adic field and $K:= k(x_1,\hdots , x_n)$. Let $f_1,\hdots, f_n \in K^{\times}$ be such that $K = k(f_1,\hdots , f_n)$. Then, $K$ satisfies $C_1(\alpha)$ for every $\alpha \in \mathrm{im}\left(\KK_2(k(f_1)) \to \KK_2(K)\right)$.
        \end{prop}
        \noindent The previous proposition applies for $f_1=x_1$, but also for 
        \[f_1 = \frac{ax_i+b}{cx_i +d} + \beta,\]
        where $i \in \{ 1,\hdots , n \}$ is an index, $a,b,c,d \in \mathbf{Q}_p$ are elements satisfying $ad - bc \neq 0$, and $\beta~\in~\mathbf{Q}_p(x_1,\hdots , x_{i-1},x_{i+1},\hdots , x_n)^{\times}$.\par
        More generally, let $k$ be a $p$-adic field and $X$ a geometrically integral variety over $k$. Denote by $K$ the function field of $X$. Fix an element $f\in K^{\times}$ and denote by $\tilde{k}$ the algebraic closure of $k(f)$ in $K$. The field $\tilde{k}$ is the function field of a smooth projective geometrically integral curve $C$ defined over an extension $k_0$ of $k$. We can repeat the argument in the proof of Proposition \ref{prop p-adic function field C2n+1} by replacing \cite[Main Theorem 1]{DiegoLuco2024KKp-adicFunction} by \cite[Corollary 4.9]{DiegoLuco2024KKp-adicFunction} to obtain the following:
        \begin{prop}
            Let $k$ be a $p$-adic field and $X$ a geometrically integral variety over $k$ of dimension $n$. Denote by $K$ the function field of $X$. Fix an element $f\in K^{\times}$ and denote by $\tilde{k}$ the algebraic closure of $k(f)$ in $K$. Let $k_0/k$ be a finite extension and $C/k_0$ a geometrically integral, smooth, projective curve over $k_0$ such that $\tilde{k} = k_0(C)$. Assume that $C$ admits a $0$-cycle of degree $1$ over the maximal unramified extension of $k_0$. Then $K$ satisfies $C_{n}(\alpha)$ for every $\alpha \in \mathrm{im}\left( \KK_2(\tilde{k})\to \KK_2(K)\right)$.
        \end{prop}
        \noindent A particular case where this proposition does not apply is $f=1+px^3 + p^2y^3$ in $K=\mathbf{Q}_p(x,y)$ because the curve defined by $z^3 + px^3 + p^2y^3$ in $\P^2_{\mathbf{Q}_p}$ does not admit a zero-cycle of degree $1$ over $\mathbf{Q}_p^{nr}$.\par
        We can also apply Theorem~\ref{thm non-contant transfer} and Proposition~\ref{prop non-constant transfer Laurent} in this context to deduce the following result.
        \begin{prop}
            Let $k$ be a $p$-adic field and $X$ a geometrically integral variety over $k$ of dimension $n$. Denote by $K$ the function field of $K$. Then $K$ satisfied $C_{n+1}(\alpha)$ for every $\alpha \in k^{\times}$ and $k\lau{t_1}\dots\lau{t_n}$ satisfies $C^{n+1}_{n+1}$.
        \end{prop}

    \subsection{Totally imaginary number fields}
    \noindent In the case of function fields over a totally imaginary number field $k$ we cannot apply Theorem~\ref{thm transfer applications} because nothing is known in this case regarding the $C_i^q$ properties for $k(x)$ when $i\neq 0$. We still get some partial result using \cite[Théorème~6.1]{Wittenberg2015KKQp} and Theorem~\ref{thm constant transfer}.
    \begin{prop}
        Let $k$ be a totally imaginary number field and $X$ a geometrically integral variety over $k$ of dimension $n$. Denote by $K$ the function field of $X$. Then $K$ satisfies $C_{n+1}(\alpha)$ for every $\alpha \in k^{\times}$. 
    \end{prop}
    \noindent We can also consider Laurent series over $k$. In this context, Proposition \ref{prop non-constant transfer Laurent} gives the following result:
    \begin{prop}
        Let $k$ be a totally imaginary number field. Then $k(\!(t_1)\!) \hdots (\!(t_n)\!)$ satisfies $C_{n+1}^{n+1}$.
    \end{prop}
    \noindent Given Kato-Kuzumaki's conjecture, this is far from the expected property. However, the only previously known Kato-Kuzumaki's property for such a field was $C_0^{n+2}$ (the property given by the norm-residue isomorphism).

    \appendix
    \section{Some cases of Kato-Kuzumaki in positive characteristic}
    \noindent In this appendix we prove Kato-Kuzumaki's conjecture for local and global fields in positive characteristic. We will also prove Kato-Kuzumaki's conjecture for separably closed fields.
    \subsection{Local fields}\label{Appendix Local fields}
    \noindent It is known that $\mathbf{Q}_p$ satisfies $C_1^1$ and $\mathbf{F}_p\lau{t}$ satisfies $C_1^1$ \textit{away from} $p$, see \cite[Corollaire~5.5 and Corollaire~4.7]{Wittenberg2015KKQp}. By applying the results in \cite{Kostas2024GeometricallyC1Fields} we can prove that $\mathbf{F}_p(\!(t)\!)$ actually satisfies $C_1^1$ and give a new proof in the case of $\mathbf{Q}_p$.
    \begin{prop}\label{prop local fields B11}
        The fields $\mathbf{F}_p\lau{t}$ and $\mathbf{Q}_p$ satisfy $C_1^1$.
    \end{prop}
    \begin{proof}
        Let $k$ be a finite extension of either $\mathbf{F}_p\lau{t}$ or $\mathbf{Q}_p$ and $Z \subseteq \mathbf{P}^n_k$ a hypersurface of degree $d$ such that $d \leq n$. It is enough to prove that every uniformiser $\pi$ of $k$ belongs to $\mathrm{N}_1(Z/k)$.\par
        Let $K_{\pi}/k$ be the maximal abelian totally ramified extension associated to $\pi$ through Lubin-Tate theory and $K/k$ a maximal totally ramified extension containing $K_{\pi}$. The field $K$ is $C_1$, see \cite[Corollary 1.1.3]{Kostas2024GeometricallyC1Fields}. Then we can find a finite subextension $k'/k$ of $K/k$ such that $Z(k') \neq \emptyset$. We have $\mathrm{N}_{k'/k}(k'^{\times}) = \mathrm{N}_{k' \cap K_{\pi}/}(k' \cap K_{\pi}^{\times})$ thanks to \cite[Theorem~3.5]{milneCFT}. Then $\pi$ belongs to $\mathrm{N}_{k'/k}(k'^{\times})$ due to \cite[Corollary~5.12]{Yoshida2008CFTLubinTate}.
    \end{proof}
        \noindent In the previous proof, the hypersurface $Z$ could be replaced by a smooth projective geometrically rationally connected variety because \cite[Corollary 1.1.3]{Kostas2024GeometricallyC1Fields} also states that $K$ is \textit{geometrically} $C_1$, giving the following result.
        \begin{prop}
            Let $k$ be a classical local field, i.e. a finite extension of either $\mathbf{F}_p\lau{t}$ or $\mathbf{Q}_p$. Then, for every smooth projective geometrically rationally connected variety $X$ we have $k^{\times}=\mathrm{N}_1(X/k)$.
        \end{prop}

    \subsection{Global fields in positive characteristic} \label{appendix global function fields}
        \noindent Let us recall the definition of \textit{separable norm groups}, \cite[Definition 1.3]{Diego2018KK}.
    \begin{defi}
        Let $k$ be a field and $Z$ a variety over $k$. For every non-negative integer $q$ we define the group $\mathrm{N}^s_q(Z/k)$ as follows
        \[
            \mathrm{N}^s_q(Z/k) = \left\langle \mathrm{N}_{q}(l_s/k) \, | \, Z(l)\neq\emptyset\right\rangle,
        \]
        where $l_s/k$ is the maximal separable subextension of $l/k$.
    \end{defi}
    \noindent These groups can differ from the norm group $\mathrm{N}_q(Z/K)$, see Remark \ref{rmk hypersup with only non-simple points}. However, the following proposition ensures that there is no difference as soon as the $p$-degree is small enough.
    \begin{prop} \label{prop separable norms coincide with norms}
        Let $K$ be a field such that $[K:K^p] = p^{\delta}$. Then for any variety $Z$ over $K$ and $q\geq \delta$ we have $\mathrm{N}_q(Z/K) = \mathrm{N}_q^s(Z/K)$.
    \end{prop}
    \noindent This follows from the following lemma.
    \begin{lemma}\label{lemma insep above p-degree surjective}
         Let $K$ be a field such that $[K:K^p] = p^{\delta}$. Then for any purely inseparable finite extension $L/K$ we have $\mathrm{N}_q(L/K) = \KK_q(K)$ for every $q \geq \delta$.
    \end{lemma}
    \begin{proof}
        Assume that $L=K(a^{1/p})$ for some $a \in K\setminus K^p$. Note that the hypersurface $Z$ of $\P^1_K$ defined by $X_0^p + aX_1^p$ only has point in extensions containing $L$. We can apply \cite[Proposition 2 (1)]{KatoKuzumaki1986Dimension} to $Z$ to deduce that $\mathrm{N}_q(L/K) = \KK_q(K)$.\par
        A general extension can be decomposed as a tower $K= K_0 \subseteq \hdots \subseteq K_{r-1} \subseteq K_r =L$ such that $K_i /K_{i-1}$ is of the form $K_i = K_{i-1}(a_i^{1/p})$ for some $a_i \in K_{i-1}$. We conclude by the previous paragraph and transitivity of the norm.
    \end{proof}
    \begin{proof}[Proof Proposition \ref{prop separable norms coincide with norms}]
        Let $l/k$ be a finite extension such that $Z(l)\neq \emptyset$. By Lemma~\ref{lemma insep above p-degree surjective} we have that $\mathrm{N}_{q}(l/l_s) = \KK_q(l_s)$, where $l_s/k$ is the maximal separable subextension of $l$. In particular it follows that $\mathrm{N}_q(Z/k) = \mathrm{N}_q^s(Z/k)$ by transitivity of the norm.
    \end{proof}
    \noindent For the applications to functions fields over finite fields we only need the case where $\delta = q = 1$ which admits an elementary proof that we include here for completeness.
        \begin{lemma}\label{lemma norm of purely inseparable extension pdeg one}
        Let $K$ be a field of positive characteristic $p$ such that $[K:K^p] = p$ and $L/K$ a purely inseparable extension. Then the norm map $\mathrm{N}_{L/K}:L^{\times} \to K^{\times}$ is surjective.
    \end{lemma}
    \begin{proof}
        Assume that $L = K(a^{1/p})$ for some $a\in K^{\times} \setminus K^{\times p}$. Since $[K:K^p]=p$, we have that $K = K^p(a)$, and hence for every $x \in K$ there exists elements $\lambda_0, \hdots , \lambda_{p-1} \in K$ such that 
        \[x = \sum_{i =0} ^{p-1} \lambda_i^p a^{i}.\]
        Since the extension $L/K$ is a purely inseparable extension of degree $p$, fixing
        \[ y := \sum_{i =0} ^{p-1} \lambda_i a^{i/p} \in L\]
        we have $\mathrm{N}_{L/K}(y) = y^p= x$. For a general extension, we can find  a tower  of subextension $K=K_0 \subseteq \hdots \subseteq K_{r-1}\subseteq K_r = L$ such that $K_i/K_{i-1}$ is of the form $K_i = K_{i-1}(a_i^{1/p})$ for some $a_i \in K_{i-1}$. We conclude by transitivity of the norm.
    \end{proof}
    \noindent Now that we have Proposition~\ref{prop separable norms coincide with norms}, \cite[Theorem 1.4]{Diego2018KK} directly implies the following:
    \begin{thm}\label{thm Diego}
        Let $K$ be a global function field and $Z$ a variety over $K$ containing a geometrically irreducible closed subvariety. Then,
        \[
        \ker \left(K^{\times}/ \mathrm{N}_1(Z/K) \to \prod_{v \in \Omega_K}K^{\times}_v/ \mathrm{N}_1(Z \otimes_K K_v / K_v) \right) = 0.
        \]
    \end{thm}
    \noindent By following the proof of \cite[Corollary 1.9]{Diego2018KK}, we obtain the following:
    \begin{prop}\label{prop global fields C11}
        The field $\mathbf{F}_p(x)$ satisfies $C_1^1$.
    \end{prop}
    \noindent We include a proof for completeness.
    \begin{proof}
        Let $K/\mathbf{F}_p(x)$ be a finite extension and $Z \subseteq \P^n_K$ a hypersurface of degree $d$ such that $d\leq n$. Since $\chi(Z, \mathcal{O}_Z) =1$, we can apply \cite[Proposition 3.3]{Wittenberg2015KKQp} to find finite extensions $K_1,\hdots ,K_m/K$ such that $Z \otimes_K K_i$ contains a geometrically irreducible closed subvariety and $\mathrm{gcd}\{[K_i:K] \, | \, i=1,\hdots , m\} =1$. Moreover, Proposition \ref{prop local fields B11} and Theorem \ref{thm Diego} ensures that $K_i^{\times} = \mathrm{N}_1(Z\otimes_K K_i /K_i)$. A standard restriction-corestriction argument allows us to conclude.
    \end{proof}
    \subsection{Separably closed fields}\label{Appendix KK separably closed field}
    \noindent In this section we prove Kato-Kuzumaki's conjecture for separably closed fields. Surprisingly, this proof is fairly elementary and was not considered in the literature before. We need the following well known facts.
    \begin{prop}\label{prop norm purely inseparable}
        Let $k$ be a field of characteristic $p$ and $l/k$ a finite purely inseparable extension. Then, for every non-negative integer $q$ the group $\mathrm{K}_q(k)$ is $p$-torsion free and the natural map $i_{l/k}:\mathrm{K}_q(k) \to \mathrm{K}_q(l)$ is injective. Also, the norm map $\mathrm{N}_{l/k}:~\mathrm{K}_q(l)~\to~\mathrm{K}_q(k)$ is given by $\mathrm{N}_{l/k}(\alpha) = \alpha^{[l:k]}$.
    \end{prop}
    \begin{proof}
        For a proof of the fact that $\mathrm{K}_q(k)$ is $p$-torsion free, see \cite[Theorem 8.1]{GeisserLevine2000Kthrycharp}. We deduce that $i_{l/k}$ is injective because the composition $\mathrm{N}_{l/k} \circ i_{l/k}$ is multiplication by $[l:k]$ and $\KK_q(k)$ is $p$-torsion free. \par
        Note that in order to prove the statement about the norm, it is enough to check it for symbols. Consider $a_1,\dots, a_q \in l^{\times}$ and the symbol $\alpha = \{a_1 , \hdots , a_q \} \in \mathrm{K}_q(l)$. Since $l/k$ is purely inseparable, there exists an integer $r \in \mathbf{N}$ such that $a^{p^{r}}_i \in k$ for every $i=1,\hdots q$. In particular, $\alpha^{p^{rq}}$ belongs to $\mathrm{K}_q(k)$. We then have 
        \begin{align*}
            \mathrm{N}_{l/k}(\alpha)^{p^{qr}} & = \mathrm{N}_{l/k}(\alpha^{p^{qr}}) \\
            &= \alpha^{p^{qr}[l:k]} \\
            &= (\alpha^{[l:k]})^{p^{qr}}.
        \end{align*}
        We conclude that $\mathrm{N}_{l/k}(\alpha) = \alpha^{[l:k]}$ because $\mathrm{K}_q(l)$ is $p$-torsion free.
    \end{proof}
    \begin{prop}\label{prop KK separably closed}
            Let $k$ be a separably closed field. Assume that $[k:k^p]\leq p^{\delta}$ for some $\delta \in \mathbf{N}$. Then for every pair of non-negative integers $i,q$ such that $i+q = \delta$ the field $k$ satisfies $C_i^q$. In particular, $k$ satisfies Kato-Kuzumaki's conjecture.
        \end{prop}
        \begin{proof}
            Let us fix $i$ for the rest of the proof and proceed by induction on $q$. The case $q = 0$ corresponds to Proposition~\ref{prop separable closed Cdelta}. Let us write our inductive hypothesis explicitly: we assume that for every separably closed field $K$ such that $[K:K^p]\leq p^{\delta - 1}$, finite extension $K'/K$, symbol $\alpha \in \KK_{q-1}(K')$, and hypersurface $Z \subseteq \P^n_{K'}$ of degree $d$ such that $d^i\leq n$, there exists a finite extension $L/K'$ such that $Z'(L)\neq \emptyset$ and $\alpha$ belongs to $\mathrm{N}_{q-1}(L/K')$. \par
            
            Let $l/k$ be a finite extension and $Z \subseteq \mathbf{P}^n_l$ be hypersurface of degree $d$ such that $d^i \leq n$. It is enough to prove that every symbol $\alpha = \{a_1, \hdots , a_q\} \in \mathrm{K}_q(k)$ belongs to $\mathrm{N}_q(Z/l)$. We may assume that $a_1 \not \in k^p$. Indeed, $\alpha$ belongs to $\mathrm{N}_q(Z/l)$ as soon as $\{a^{1/p}_1, a_2,\hdots , a_q \}$ does. In the case that $a_1$ is a $p^r$-power for every $r\in \mathbf{N}$, the symbol $\alpha$ belongs to $\mathrm{N}_q(l'/l)$ for every finite extension $l'/l$ due to Proposition \ref{prop norm purely inseparable}. For $r \in \mathbf{N}$ we denote by $k_r$ the field $k(a_1^{1/p^r})$ and fix
            \[ K := \bigcup_{r \geq 1} k_r.\]
            Note that $[K:K^p] \leq p^{\delta -1}$, see \cite[Lemma 3.15]{DiegoLuco2023TransferSerreII}. We can apply our inductive hypothesis to find a finite extension $L/K$ and $\beta \in \mathrm{K}_{q-1}(L)$ such that $Z(L) \neq \emptyset$ and $\mathrm{N}_{L/K}(\beta) = \{a_2,\dots , a_q\}|_{K}$. Let $b_1,\hdots ,b_n \in L$ be a generating set for the extension $L/K$ and denote by $p_1(T), \hdots , p_n(T) \in K[T]$ their corresponding minimal polynomials. For a big enough integer $r_0 \in \mathbf{N}$ we may assume that $p_1(T), \hdots,p_n(T) \in k_{r_0}[T]$, that $Z(k_{r_0}(b_1,\hdots , b_n)) \neq \emptyset$, and $\beta$ comes from $\beta_0 \in \mathrm{K}_{q-1}(k_{r_0}(b_1,\hdots ,b_n))$. After noting that the tensor product $k_{r_0}(b_1,\dots , b_n) \otimes_{k_{r_0}} K$ is a field, we can apply \cite[Lemma 7.3.6]{GilleSzamuely2017CentralGalois} to deduce the following commutative diagram
            \begin{equation*}
                \begin{tikzcd}[row sep = 7ex , column sep = 13ex]
                    \mathrm{K}_{q-1}(k_{r_0}(b_1,\hdots b_n)) \ar[r,"\mathrm{N}_{k_{r_0}(b_1,\hdots ,b_n)/k_{r_0}}"] \ar[d, "\iota"]&  \mathrm{K}_{q-1}(k_{r_0}) \ar[d, "\iota"] \\
                    \mathrm{K}_{q-1}(L) \ar[r, "\mathrm{N}_{L/K}"]& \mathrm{K}_{q-1}(K).
                \end{tikzcd}
            \end{equation*}
            Since the vertical arrows are injective due to Proposition \ref{prop norm purely inseparable}, we deduce that 
            \[
            \mathrm{N}_{k_{r_0}(b_1,\hdots ,b_n)/k_{r_0}}(\beta_0) = \{a_2,\hdots , a_q\}.
            \]
            We conclude that $\alpha$ is a norm because
            \begin{align*}
                \mathrm{N}_{k_{r_0}(b_1,\hdots ,b_n)/k}(\{a_1^{1/p^{r_0}} , \beta _0 \}) &=  \mathrm{N}_{k_{r_0}/k} \circ \mathrm{N}_{k_{r_0}(b_1,\hdots ,b_n)/k_{r_0}}(\{a_1^{1/p^{r_0}} , \beta _0 \}) \\
                &=  \mathrm{N}_{k_{r_0}/k} \left( \{a_1^{1/p^{r_0}},  \mathrm{N}_{k_{r_0}(b_1,\hdots ,b_n)/k_{r_0}}(\beta_0) \}  \right) \\
                &=  \mathrm{N}_{k_{r_0}/k}(\{ a_1^{1/p^{r_0}}, a_2,\hdots , a_q\}) \\
                &= \{a_1 ,\hdots , a_q \}.
            \end{align*}
            \vspace*{-1ex}
        \end{proof}
\bibliographystyle{alpha}
\bibliography{ref.bib}
\end{document}